\documentclass{tran-l}
\usepackage[all,2cell]{xy}\UseAllTwocells
\usepackage{graphicx}
\usepackage{epstopdf}
\usepackage{pb-diagram}
\usepackage{amssymb, amscd}
\usepackage{amsmath}
\usepackage{amsfonts}
\usepackage{graphics}
 \xyoption{arc}

\newtheorem{theorem}{Theorem}[section]
\newtheorem{lemma}[theorem]{Lemma}

\newtheorem{prop}[theorem]{Proposition}
\newtheorem{corollary}[theorem]{Corollary}

\theoremstyle{definition}
\newtheorem{definition}[theorem]{Definition}
\newtheorem{example}[theorem]{Example}
\newtheorem{examples}[theorem]{Examples}

\theoremstyle{remark}
\newtheorem{remark}[theorem]{Remark}

\numberwithin{equation}{section}

\def\st{\stackrel}

\def\r{\rightarrow} 

\def\rr{\Rightarrow} 

\makeatletter
\def\Ddots{\mathinner{\mkern1mu\raise\p@
\vbox{\kern7\p@\hbox{.}}\mkern2mu
\raise4\p@\hbox{.}\mkern2mu\raise7\p@\hbox{.}\mkern1mu}}
\makeatother

\newcommand{\F}{{\mathcal F}}
\newcommand{\G}{\mathcal G}
\newcommand{\cL}{\mathcal L}
\newcommand{\Hol}{\mbox{\rm Hol}}

\newcommand{\id}{{\scriptstyle id}}
\newcommand{\Pair}{\mbox{\rm Pair}}
\newcommand{\N}{{\bf N}}

\newcommand{\cA}{{\mathcal A}}
\newcommand{\cB}{{\mathcal B}}
\newcommand{\cC}{{\mathcal C}}
\newcommand{\cD}{{\mathcal D}}
\newcommand{\cE}{{\mathcal E}}

\newcommand{\cH}{{\mathcal H}}
\newcommand{\cJ}{{\mathcal J}}
\newcommand{\cK}{{\mathcal K}}
\newcommand{\I}{{\mathcal I}}
\newcommand{\cN}{{\mathcal N}}
\newcommand{\M}{{\mathcal M}}
\newcommand{\cO}{{\mathcal O}}
\newcommand{\cP}{{\mathcal P}}

\newcommand{\cU}{{\mathcal U}}
\newcommand{\V}{{\mathcal V}}
\newcommand{\U}{{\mathcal U}}
\newcommand{\cW}{{\mathcal W}}
\newcommand{\X}{{\mathcal X}}
\newcommand{\Y}{{\mathcal Y}}

\newcommand{\BB}{{\mathfrak B}}

\newcommand{\DD}{{\mathfrak D}}

\newcommand{\HH}{{\mathfrak H}}
\newcommand{\R}{{\mathbb{R}}}
\newcommand{\Z}{{\mathbb{Z}}}

\newcommand{\cat}{\mbox{\rm cat}}

\newcommand{\simM}{~ {\sim}_{M} ~}
\newcommand{\simT}{~ {\sim}_{T} ~}

\newcommand{\al}{\alpha}
\newcommand{\e}{\epsilon}

\newcommand{\ri}{\Rightarrow}

\newcommand{\po}{\bullet^K}

\newcommand{\globe}{
\[
 \xy
  (-15,0)*{\bullet}="1";
  (0,0)*{\bullet}="2";
 "1";"2" **\crv{(-12,9) & (-3,9)};
  "1";"2" **\crv{(-12,-9) & (-3,-9)};
    (-7.5,6.75)*{\scriptstyle >}+(0,3)*{\scriptstyle \phi};
    (-7.5,-6.75)*{\scriptstyle >}+(0,-3)*{\scriptstyle \psi};
    (-17.5,0)*{\scriptstyle \cK};
  (2.5,0)*{\scriptstyle \G};
  {\ar@{=>}_{\scriptstyle a}(-7.5,3)*{};(-7.5,-3)*{}} ;
 \endxy 
 \]
}
\newcommand{\vglobe}{
 \[
 \xy
  (-8,0)*{\bullet}="1";
  (8,0)*{\bullet}="2";
 "1";"2" **\crv{(-5,9) & (5,9)};
  "1";"2" **\crv{(-5,-9) & (5,-9)};
   "1";"2" **\dir{-} ?(.55)*\dir{>};
    (,6.75)*{\scriptstyle >}+(0,3)*{\scriptstyle \phi};
    (0,-6.75)*{\scriptstyle >}+(0,-3)*{\scriptstyle \psi};
    (-10.5,0)*{\scriptstyle \cK};
  (10.5,0)*{\scriptstyle \G};
  {\ar@{=>}_{b}(0,5)*{};(0,1.5)*{}} ;
  {\ar@{=>}_{a}(0,-1.5)*{};(0,-5)*{}} ;
 \endxy.
 \]
}
\newcommand{\hglobe}{
 \[
 \xy
  (-15,0)*{\bullet}="1";
  (0,0)*{\bullet}="2";
  (15,0)*{\bullet}="3";
 "1";"2" **\crv{(-12,9) & (-3,9)};
  "1";"2" **\crv{(-12,-9) & (-3,-9)};
   "2";"3" **\crv{(3,9) & (12,9)} ;
    "2";"3" **\crv{ (3,-9)& (12,-9) };
    (-7.5,6.75)*{\scriptstyle >}+(0,3)*{\scriptstyle \varphi};
    (7.5,6.75)*{\scriptstyle >}+(0,3)*{\scriptstyle \phi};
    (-7.5,-6.75)*{\scriptstyle >}+(0,-3)*{\scriptstyle \varphi'};
    (7.5,-6.75)*{\scriptstyle >}+(0,-3)*{\scriptstyle \phi'};
    (-17.5,0)*{\scriptstyle \cL};
  (17.5,0)*{\scriptstyle \G};
  {\ar@{=>}_{a}(7.5,3)*{};(7.5,-3)*{}} ;
  {\ar@{=>}_{b}(-7.5,3)*{};(-7.5,-3)*{}} ;
   (-18,0)*{};
  (18,0)*{};
 \endxy .
 \]
}
\newcommand{\il}{
\[
\xy
  (-8,0)*{\bullet}="1";
  (8,0)*{\bullet}="2";
 "1";"2" **\crv{(-5,9) & (5,9)};
  "1";"2" **\crv{(-5,-9) & (5,-9)};
   "1";"2" **\dir{-} ?(.55)*\dir{>};
    (,6.75)*{\scriptstyle >}+(0,3)*{\scriptstyle \varphi};
    (0,-6.75)*{\scriptstyle >}+(0,-3)*{\scriptstyle \varphi''};
    (-10.5,0)*{\scriptstyle \cL};
   {\ar@{=>}_{d}(0,5)*{};(0,1.5)*{}} ;
  {\ar@{=>}_{c}(0,-1.5)*{};(0,-5)*{}} ;
     (8,0)*{\bullet}="1";
  (24,0)*{\bullet}="2";
 "1";"2" **\crv{(11,9) & (21,9)};
  "1";"2" **\crv{(11,-9) & (21,-9)};
   "1";"2" **\dir{-} ?(.55)*\dir{>};
    (16,6.75)*{\scriptstyle >}+(0,3)*{\scriptstyle \phi};
    (16,-6.75)*{\scriptstyle >}+(0,-3)*{\scriptstyle \phi''};
  (26.5,0)*{\scriptstyle \G};
  {\ar@{=>}_{b}(16,5)*{};(16,1.5)*{}} ;
  {\ar@{=>}_{a}(16,-1.5)*{};(16,-5)*{}} ;
 \endxy .
 \]
}

\begin{document}

\title{The Lusternik-Schnirelmann category of a Lie groupoid}
\author{Hellen Colman}
\address{Department of Mathematics \\Wilbur Wright College \\\newline 4300 N. Narragansett Avenue \\Chicago, IL 60634\\
   USA} 
  \email{hcolman@ccc.edu}
\urladdr{http://faculty.ccc.edu/hcolman/}


\begin{abstract}
We propose a new homotopy invariant for Lie groupoids which generalizes the classical Lusternik-Schnirelmann category for topological spaces. We use a bicategorical approach to develop a notion of contraction in this context. We propose a notion of homotopy between generalized maps given by the 2-arrows in a certain bicategory of fractions. This notion is invariant under Morita equivalence. Thus, when the groupoid defines an orbifold, we have a well defined LS-category for orbifolds. We prove an orbifold version of the classical Lusternik-Schnirelmann theorem for critical points.
\end{abstract}    

\maketitle

\section{Introduction}
The LS-category of a topological space is an invariant of the  homotopy type of the space introduced in the early 1930's by Lusternik and Schnirelmann  \cite{Lu}. It is a numerical invariant that measures the complexity of the space, in particular if $M$ is a compact manifold the LS-category of $M$ provides a lower bound for the number of critical points of any smooth function on the manifold.
The LS-category of a space $X$ is defined as the least number of open subsets, contractible in $X$, required to cover $X$. For a survey in LS-category see \cite{Ja,Op}.

There are many generalizations of the original concept adapted to various 
contexts such as the {\it fibrewise category}, introduced by I.M.~James and J.R.~Morris \cite{JaMo},  the {\it equivariant 
category}, by E.~Fadell \cite{Fa} and the {\it transverse} and {\it tangential} categories for foliated manifolds \cite{tran,tan,survey}.

Our aim is to develop a Lusternik-Schnirelmann theory in the context of Lie groupoids. Our main application will be in the setting of orbifolds as groupoids.

Since an orbifold is defined as a Morita equivalence class of certain type of groupoids \cite{MP,M} our notion of LS-category for Lie groupoids ought to be invariant under Morita equivalence. In this spirit, the right notion of morphism between Lie groupoids is that of a generalized map \cite{MP,M}. We develop a notion of homotopy between generalized maps. We start with a notion of strong homotopy between strict morphisms which is not Morita invariant, and introduce the notion of {\it essential homotopy equivalence} that simultaneously weakens a strict homotopy and generalizes an essential equivalence. We will say that two Lie groupoids $\cK$ and $\G$ have the same {\it Morita homotopy type} if there exists a third groupoid $\cJ$ and essential homotopy equivalences
$$\cK\overset{\nu}{\gets}\cJ\overset{\eta}{\to}\G.$$
We prove that the bicategory $\HH$ of Lie groupoids, strict morphisms and strict homotopies admits a bicalculus of fractions that formally inverts the essential homotopy equivalences $W$. Using the techniques developed in \cite{P} we obtain a new bicategory of Lie groupoids $\HH(W^{-1})$ where the essential homotopy equivalences were formally inverted. Morita homotopy equivalences introduced above amount to {\it isomorphisms} in this bicategory.

Our notion of {\it Morita homotopy} corresponds to the 2-arrows in the bicategory $\HH(W^{-1})$. This deformation within the groupoid is closely related to the notion of $\G$-path developed by Haefliger \cite{H2,H3}. We propose a modified version of a $\G$-path (between two objects)  that we call a {\it multiple $\G$-path} (between two orbits). A Morita homotopy determines multiple $\G$-paths between the orbits. A multiple $\G$-path is defined as a generalized map $\I\overset{\e}{\gets}\I'\overset{\sigma}{\to}\G$ where $\I$ and $\I'$ are certain groupoids associated to the interval $[0,1]$.
Multiple $\G$-paths will play a key role in defining the integral curves of a gradient vector field in a Lie groupoid.

Based on the Morita homotopy notion, we define a $\G$-categorical subgroupoid in the sense of Lusternik and Schnirelmann. A subgroupoid $\U$ is $\G$-categorical if it can be deformed by a Morita  homotopy into a transitive groupoid. The {\it groupoid LS-category of $\G$}, $\cat \G$, is the minimal number of $\G$-categorical subgroupoids required to cover $\G$. When $\G$ is the unit groupoid over a manifold $M$, this number  specializes to the classical LS-category of the manifold $M$.

We prove that $\cat \G$ is an invariant of the Morita homotopy type. Therefore, $\cat \G$ is invariant under Morita equivalence and yields a well defined invariant for orbifolds. Let $\G$ be an orbifold groupoid defining the orbifold $\X$. The orbifold LS-category of $\X$ is defined as the groupoid LS-category of $\G$, $\cat \X=\cat \G$. 

The orbifold LS-category is a homotopy invariant that detects part of the complexity of the orbifold. It detects the existence of obstructions given by the twisted sectors to contract an open set, but somehow misses the ``weight'' of such obstructions.

We introduce a variant of the orbifold LS-category that we call the {\it weighted orbifold LS-category}, $w\cat \X$. This numerical invariant also generalizes the classical LS-category for manifolds, if the orbifold is a manifold $X$ then $w\cat X=\cat X$.  We prove that the weighted orbifold category is an invariant of orbifold homotopy type which relates to the orbifold category of the inertia orbifold.

Finally we give a version of the Lusternik-Schnirelmann theorem for Lie groupoids. If $\G\overset{\e}{\gets}\G'\overset{\phi}{\to}\R$ is a generalized map satisfying certain {\it deformation conditions} (D) then $$\cat\G\le \sum_{c\in \R} \cat \cK_c$$ where $\cK_c$ is the full subgroupoid over the critical points at the level $c$.

We show that if $\G$ is an {\it orbifold groupoid} defining a compact orbifold, then the deformation conditions (D) are satisfied and we have that  the LS-theorem holds for orbifolds: If $\X$ is a compact orbifold and $f\colon \X\to \R$ is an orbifold map, then $f$ has at least $\cat\X$ critical points.

When counting a finite number of critical points in an orbifold, we face the dilemma of how to count. Let $\cK$ be the critical groupoid. If $\cK$ is a finite groupoid, we can chose one isolated critical point $x_i$ in each orbit and we have that $\cK$ is equivalent to a groupoid whose set of objects is $\{x_1, \cdots, x_n\}$ and the set of arrows is the disjoint union of the isotropy groups $K_1\sqcup \cdots \sqcup K_n$  where the order of $K_i$ is $k_i$ and the number of conjugacy classes in $K_i$ is $k_i'$ (then $k_i=k'_i$ if $K_i$ is abelian). By the one hand we have that the cardinality of a groupoid of Baez and Dolan \cite{Ba} and the Euler characteristic of a finite category of Leinster \cite{Tom}  both give as a cardinality of $\cK$ the rational number $\displaystyle {1\over {k_1}}+\cdots+{1\over {k_n}}$. On the other hand the string theoretic Euler characteristic \cite{Wi} gives as a cardinality of $\cK$ the integer number $\displaystyle { {k'_1}}+\cdots+{{k'_n}}$. Our weighted orbifold category is a lower bound for this latter cardinality. 

This  paper is organized as follows. Section 2 presents the notions of strong and Morita equivalence for strict morphisms between Lie groupoids. We describe in this section some examples that will be studied further throughout the paper. Section 3 is where we develop the notion of strict homotopy associated to a subdivision of the interval $I=[0,1]$. We recall the notion of $\G$-path and introduce the multiple $\G$-paths. We introduce the notion of essential homotopy equivalence and homotopy pullback of groupoids. Section 4 is where we show that the bicategory of Lie groupoids, strict morphisms and strict homotopies constructed in section 3 admits a bicalculus of fractions that inverts the essential homotopy equivalences. We define our notion of Morita homotopy as the 2-arrows in this bicategory of fractions. In section 5 we propose a notion of LS-category for Lie groupoids. We say that a subgroupoid is $\G$-categorical, in the sense of Lusternik and Schnirelmann, if there is a Morita homotopy between the inclusion and a morphism with image in a single orbit. We define the LS-category for Lie groupoids and prove that it is invariant under Morita homotopy. In particular, it is invariant under Morita equivalence and yields a well defined number for orbifolds. In section 6 we treat in detail this LS theory for orbifolds. We describe the obstructions to deform an open set in terms of the orbifold structure of the twisted sectors and give estimates for the orbifold LS-category in terms of the categories of the twisted and untwisted sectors. We  also propose  the notion of weighted orbifold LS-category and prove that it is invariant under Morita homotopy equivalence. Section 7 is where we give the preliminaries and prove the Lusternik-Schnirelmann theorem for orbifolds. Given a generalized map  $\G\overset{\e}{\gets}\G'\overset{\phi}{\to}\R$ we introduce a groupoid version of the classical deformation conditions. When the groupoid defines an orbifold, we can define the $\G$-flow of the gradient of $(\e, \phi)$ and its integral curves will be given by multiple $\G$-paths. These integral $\G$-curves provide the tool to prove that the deformation conditions are satisfied for  smooth maps on orbifold groupoids.  We show that both $\cat \G$ and $w\cat \G$ satisfy some monotonicity, additivity, deformation invariance and continuity properties required to prove the theorem. 

\subsection*{Acknowledgements}
I am grateful to Andr\'e Haefliger, for suggesting to me the idea of developing a Lusternik-Schnirelmann theory for orbifolds. My thanks also to John Oprea and Dorette Pronk for their valuable comments and help.

\section{Recollections on Lie groupoids and equivalences}

\subsection{Lie groupoids}
Recall that a  {\it groupoid} $\G$ is a small category in which each arrow is invertible \cite{McL} \cite{Brown}. 
Our notation for groupoids is that
$G_0$ is the space of objects and $G_1$ is the space of objects, with source and target maps
$s,t:G_1\to G_0$, multiplication $m:G_1 \times_{G_0} G_1\to G_1$, inversion $i:G_1\to G_1$, and object inclusion
$u:G_0\hookrightarrow G_1$.

A Lie groupoid is a groupoid for which $G_1$ and $G_0$ are manifolds,
$s$ and $t$ are surjective submersions, and $m$, $i$ and $u$ are smooth \cite{Mck}.

The set of arrows from $x$ to $y$ is denoted $G(x,y)=\{ g\in G_1 | s(g)=x \mbox{ and } t(g)=y\}$. The set of arrows from $x$ to itself, $G(x,x)$, is a group called the {\it isotropy} group of $G$ at $x$ and denoted by $G_{x}$. The {\it orbit} of $x$ is the set ${\mathcal{O}}(x)=ts^{{-1}}(x)$. The orbit space $|\G|$ of $\G$ is the quotient of $G_{0}$ under the equivalence relation: $x\sim y$ iff $x$ and $y$ are in the same orbit.

We will describe now various Lie groupoids that will appear throughout the paper.
\begin{enumerate}
\item
{\it Unit groupoid.} 
Consider the groupoid $\G$ associated to a manifold $M$ with $G_0=G_1=M$. This is a Lie groupoid whose arrows are all units, called the unit groupoid and denoted $\G=u(M)$.
\item
{\it Pair groupoid.} Let $\G$ with $G_0=M$ as before and consider $G_1=M\times M$. This is a Lie groupoid with exactly one arrow from any object $x$ to any object $y$, called the pair groupoid and denoted $\G=\Pair(M)$. 
\item
{\it Point groupoid.} Let $G$ be a Lie group. Let $\bullet$ be a point. Consider the groupoid $\G$ with $G_0=\bullet$ and $G_1=G$. This is a Lie groupoid with exactly one object $\bullet$ and $G$ is the manifold of arrows in which the maps $s$ and $t$ coincide. We denote the point groupoid by $\bullet^G$.
\item
{\it Translation groupoid.} Let $K$ be a Lie group acting (on the left) on  a manifold $M$. Consider the groupoid $\G$ with $G_0=M$ and $G_1=K\times M$. This is a Lie groupoid with arrows $(k, x)$  from any object $x$ to $y=kx$, called the translation or action groupoid and denoted $\G=K\ltimes M$. 
\item
{\it Holonomy groupoid.} Let $(M,\F)$ be a foliated manifold. Consider the groupoid $\G$ with $G_0=M$ and whose arrows from $x$ to $y$ on the same leaf $L\in\F$ are the holonomy classes of paths in $L$ from $x$ to $y$. There are not arrows between points in different leaves. This is a Lie groupoid called the holonomy groupoid and denoted $\G=\Hol(M,\F)$.
\end{enumerate}
We will denote ${\bf{1}}$ the trivial groupoid with one object and one arrow, ${\bf{1}}=u(\bullet)=\Pair(\bullet)=\bullet^{\bullet}$.

\subsection{Strong and Morita equivalence}
From now on all groupoids will be assumed to be Lie groupoids.

A {\it morphism} $\phi: \cK \to \G$ of groupoids is a functor given by two smooth maps $\phi: K_1 \to G_1$ and $\phi: K_0 \to G_0$ that together commute with all the structure maps of the groupoids $\cK$ and $\G$.

A {\it natural transformation} $T$ between two morphisms $\phi, \psi: \cK \to \G$ is a smooth map $T: K_{0} \to G_1$ with $T(x):\phi(x)\to\psi(x)$ such that for any arrow $h:x \rightarrow y$ in $K_1$,  the identity $\psi(h)T(x)=T(y)\phi(h)$ holds. We write $\phi\simT\psi$.

A morphism $\phi: \cK \to \G$ of groupoids is an {\it equivalence} of groupoids if there exists a morphism $\psi: \G \to \cK$ of groupoids and natural transformations $T$ and $T'$ such that $\psi\phi\sim_{T} \id_{\cK}$ and $\phi\psi\sim_{T'} \id_{\G}$. Sometimes we will refer to this notion of equivalence as a ${\it strong}$ equivalence.

A morphism $\e: \cK \to \G$ of groupoids is an {\it essential equivalence} of groupoids if
\begin{itemize}
\item[(i)] $\e$ is essentially surjective in the sense that \[t\pi_{1}:G_1\times_{G_0}K_0\rightarrow G_0\] is a surjective submersion
where $G_1\times_{G_0}K_0$ is the pullback along the source $s: G_1\to G_0$;
\item[(ii)] $\e$ is fully faithful in the sense that $K_1$ is given by the following pullback of manifolds:
\[\xymatrix{
K_1 \ar[r]^{\e} \ar[d]_{(s,t)}& G_1 \ar[d]^{(s,t)} \\ 
K_0\times K_0 \ar[r]^{\e \times \e} & G_0\times G_0}\]
\end{itemize}
The first condition implies that for any object $y\in G_0$, there exists an object  $x\in K_0$ whose image $\e(x)$ can be connected to $y$ by an arrow $g\in G_1$. The second condition implies that for all $x,z\in K_0$, $\e$ induces a diffeomorphism $K(x,z)\approx G(\e(x),\e(z))$ between the submanifolds of arrows.

For general categories the notions of equivalence and essential equivalence coincide. This applies to the particular case in which the categories are groupoids. But when some extra structure is involved (continuity or differentiability) these two notions are not the same anymore. An essential equivalence implies the existence of the inverse functor using the axiom of choice but not the existence of a {\it smooth} functor. 
\begin{prop}\cite{MM}
Every  equivalence of Lie groupoids is an essential equivalence.
\end{prop}

\begin{remark}
The converse does not hold for Lie groupoids. 
\end{remark}

{\it Morita equivalence} is the smallest equivalence relation between Lie groupoids such that they are equivalent whenever there exists an essential equivalence between them.
\begin{definition}
Two Lie groupoids $\cK$ and $\G$ are {\it Morita equivalent} if there exists a Lie groupoid $\cJ$ and essential equivalences
\[\cK\overset{\e}{\gets}\cJ\overset{\sigma}{\to}\G.\]
\end{definition}
This defines an equivalence relation that we denote $\cK\simM \G$. In this case, it is always possible to chose the equivalences $\e$ and $\sigma$ being surjective submersions on objects \cite{MM}.
\begin{example} \label{tear} Consider the Lie groupoid $\G$ whose manifold of objects is the disjoint union of two open disks, $G_0=D_1\sqcup D_2$ as shown in figure \ref{t}. The disk $D_1$ is acted on by $\Z_3=\{1, \rho, \rho^2\}$. The manifold of arrows is the disjoint union of three copies of $D_1$, two annuli $A$ and $A'$ and a copy of $D_2$, $G_1=\displaystyle{\left(\bigsqcup_{i=1}^3 D_1^i\right)} \sqcup A\sqcup A'\sqcup D_2$. The source map $s\colon G_1\to G_0$ is given by the projection $ D_1^i\to D_1\subset G_0$,  the inclusions 
$ A\to D_1\subset G_0$ and $ D_2\to D_2\subset G_0$ and on the other annulus $A'$ is given by the triple covering $A'\to A_2\subset D_2\subset G_0$. Target map $t\colon G_1\to G_0$ coincides with source map on $D_2$ whilst  in $\displaystyle{\left(\bigsqcup_{i=1}^3 D_1^i\right)}$ it is given by the action of $\Z_3$, $\rho^i\colon D_1^i\to D_1\subset G_0$ and in the annuli $A$ and $A'$ is given by the triple covering $A\to A_2\subset D_2\subset G_0$ and the inclusion $ A\to D_1\subset G_0$ respectively.

\begin{figure}
\includegraphics[height=3in]{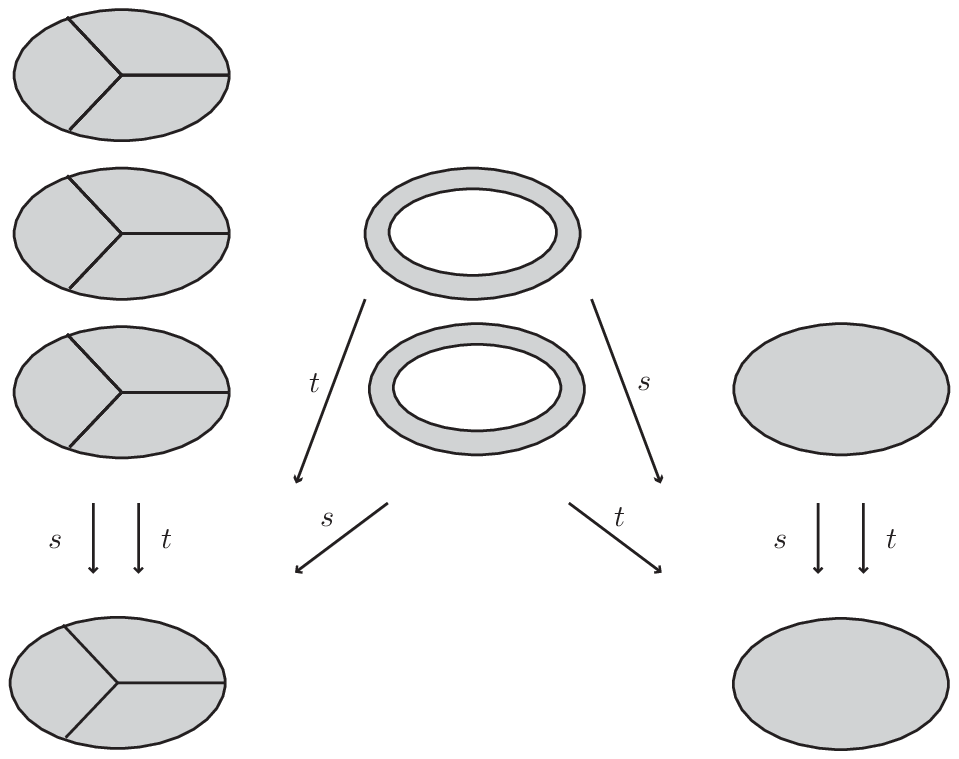}
\caption{}
\label{t}
\end{figure}

Now consider the translation groupoid $\cK=S^1\ltimes S^3$ defined by the action of $S^1$ on $S^3$ given by  $(v,w)\mapsto (z^3v,zw)$. The  manifold of objects is $G_0=S^3$ and the manifold of arrows is  $G_1=S^3\times S^1$.
If we think of $S^3$ as the union of two solid tori as shown in figure \ref{solid}, we have that the orbits of this action are circles of length $2\pi$ for points on the two cores of the tori and circles of length $2\pi\sqrt{3\|v\|^2+\|w\|^2}$ elsewhere. This is a Seifert fibration on $S^3$. The inclusion of two transverse disks to this  Seifert fibration determines a functor $\e\colon \G\to \cK$ which is an essential equivalence. Thus the groupoids $\cK$ and $\G$ are Morita equivalent.

\begin{figure}
\includegraphics[height=2in]{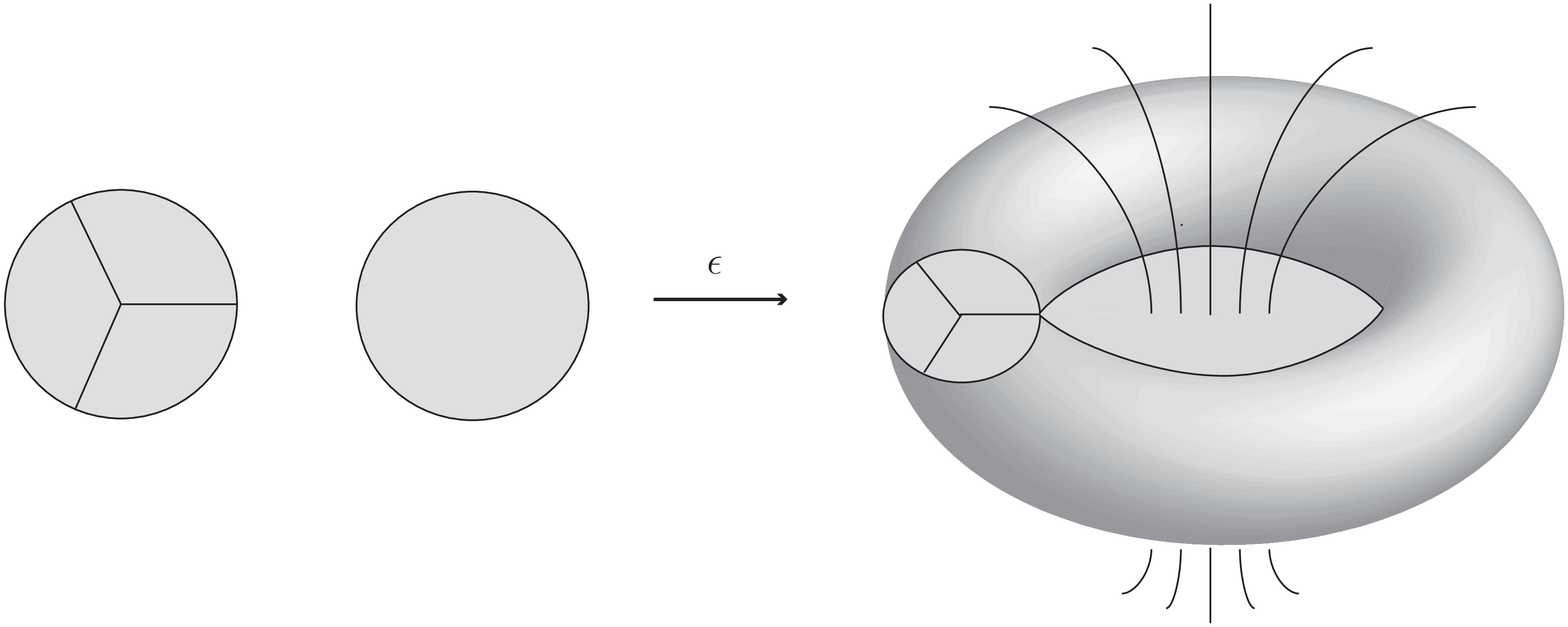}
\caption{}
\label{solid}
\end{figure}
\end{example}

\section{Homotopy}
We will define now two notions of homotopy that will generalize the notions of strong and Morita equivalence respectively.

Given a subdivision $S=\{0=r_0\le r_1\le\cdots\le r_n=1\}$ of the interval $I=[0,1]$ consider  a groupoid $\I_S$ whose manifold of objects  is given by the following disjoint union:
$$\bigsqcup_{i=1}^{n} [r_{i-1}, r_{i}]$$
An element in the connected component  $[r_{i-1}, r_{i}]$ will be denoted by $(r,i)$. Then the manifold of objects  is ${(\I_S)}_0=\{(r,i)| r\in [r_{i-1}, r_{i}], i=1, \ldots, n\}$.
The manifold of arrows of $\I_S$ is given by the disjoint union: $$\left(\bigsqcup_{i=1}^{n}  [r_{i-1}, r_{i}]\right)\bigsqcup\{r_1,\cdots, r_{n-1}, r'_1,\cdots, r'_{n-1}\}$$
where $\displaystyle\bigsqcup_{i=1}^{n}  [r_{i-1}, r_{i}]$ is the set of unit arrows and for each point $r_i$ in the subdivision $S$ two arrows were added: $r_i$ and its inverse arrow $r'_i$ such that the source of $r_i$ is $(r_i,i)$ and its target is $(r_i,i+1)$.

\[
\xy
 (7,0)*{[}="2";
  (21,0)*{]}="3";
   (35,0)*{}="4";
   (28,0)*{[}="5";
  (42,0)*{]}="6";
    "3";"5" **\crv{(24,4) & (25,4)};
  "3";"5" **\crv{(24,-4) & (25,-4)};
  (24.5,3)*{\scriptstyle >}+(0,3)*{\scriptstyle r_i};
    (24.5,-3)*{\scriptstyle <}+(0,-3)*{\scriptstyle r^{-1}_i};
      (7,-4)*{\scriptstyle r_{i-1}};
    (21,-4)*{\scriptstyle r_i};
  (14,2.5)*{\scriptstyle };
     (28,-4)*{\scriptstyle r_{i}};
    (42,-4)*{\scriptstyle r_{i+1}};
   (56,2.5)*{\scriptstyle };
   "2"; "3" **\dir{-} ;
    "5"; "6" **\dir{-};
{ \ar@{}"4"; "5"};
(0,0)*{\cdots};
(49,0)*{\cdots};
\endxy 
\]
We call  ${\bf{0}}$ the trivial groupoid over $\{0\}$ and  ${\bf{1}}$ the trivial groupoid over $\{1\}$.

Let $\phi,\psi: \cK \to \G$ be morphisms. We will say that $\phi$ is {\it homotopic} to $\psi$ if there exists a subdivision $S$ and a morphism $H^S: \cK\times \I_S \to \G$ such that 
$H^S_{\bf{0}}= \phi$ and $H^S_{\bf{1}}= \psi$.

This defines a relation of equivalence between morphisms that we will call {\it homotopy equivalence}  and that sometimes we will refer to it as ${\it strong}$ or ${\it strict}$ homotopy equivalence depending what feature we want to emphasize since we will introduce a weaker version (in the same sense that an essential equivalence weakens a strong equivalence) and a generalized version of this notion (in the same sense that a Morita equivalence generalizes an essential equivalence).
\begin{examples}
\begin{enumerate}
\item A natural transformation $T\colon \phi\to\psi$ determines an {\it homotopy}  $H^S: \cK\times \I_S \to \G$ over the subdivision $S=\{0=r_0, r_1=1/2, r_2=1\}$. The homotopy is given by 
\begin{equation*}
H^S(x,(r,i))=
\begin{cases} \phi(x) & \text{if $r\le 1/2$ and $i=1$,}
\\
\psi(x) &\text{if $r\ge 1/2$ and $i=2$.}
\end{cases}
\end{equation*}
on objects;
\begin{equation*}
H^S(k,u(r,i))=
\begin{cases} \phi(k) & \text{if $r\le 1/2$ and $i=1$,}
\\
\psi(k) &\text{if $r\ge 1/2$ and $i=2$.}
\end{cases}
\end{equation*}
on units arrows of $\I_S$ and $H^S(k,r_1)=T(y)\phi(k)$ for the arrow $r_1\colon (1/2,1)\to (1/2,2)$.

Therefore, there exist a subdivision $S$ and a morphism  $H^S$ such that $H^S(x,0)=\phi(x)$,  $H^S(x,1)=\psi(x)$ and $H^S(k,0)=\phi(k)$,  $H^S(k,1)=\psi(k)$.

\item A morphism $H: \cK\times \I \to \G$, with $\I$ being the unit groupoid over the interval $I=[0,1]$, such that $H_0\sim \phi$ and $H_1\sim \psi$ determines a homotopy $H^S: \cK\times \I_S \to \G$ over the subdivision $S=\{r_0=0, r_1=0,r_2=0,r_3=1,r_4=1,r_5=1\}$. The original morphism  $H: \cK\times \I \to \G$ does not define an equivalence relation, since transitivity fails.

\item An ordinary homotopy $H: \cK\times \I \to \G$ such that  $H_0= \phi$ and $H_1= \psi$ determines an homotopy $H^S: \cK\times \I_S \to \G$ over the subdivision  $S=\{0=r_0, r_1=1\}$. The morphism $H: \cK\times \I \to \G$ defines an equivalence relation but fails to be invariant of Morita equivalence (it is not even invariant of strong equivalence for groupoids).
\end{enumerate}
\end{examples}

\begin{prop}\label{inj}
Let $H:\U\times{\I_{S}}\to \G$ be a homotopy of the inclusion, $H_{\bf{0}}=i_{\U}$.  Then, there is an injection $j: G_x\to G_{y_r}$ for all $y_r=H^S(x,(r,i))$.
\end{prop}
\begin{proof}
The homotopy $H^S$ restricted to the connected component $U_x\times [0,r_1]$ determines an ordinary homotopy, then $G_x$ injects into $G_{H(x, (r,1))}$ for all $r\in [0, r_1]$. The isotropy groups $G_{H(x, (r_1,1))}$ and $G_{H(x, (r_1,2))}$ coincide, since there is an arrow $r_1$ from $(r_1,1)$ to $(r_1,2)$. Then, we have a finite number of injections:
$$G_x \rightarrowtail G_{H(x, (r,1))} \rightarrowtail G_{H(x, (r_1,1))}= G_{H(x, (r_1,2))}\rightarrowtail G_{H(x, (r,2))}\ldots \rightarrowtail G_{H(x, (r_n,n))}$$
and $G_x$ injects into $G_{H^S(x,(r,i))}$ for all $r\in [r_{i-1}, r_i]$ with $i=1,\ldots, n$.
\end{proof}

We recall now the definition of $\G$-path, due to Haefliger \cite{H1,H2,H3,Bri}. For another approach see also \cite{Loop}.
A {\it $\G$-path} from $x$ to $y$ over the subdivision $S=\{0=r_0\le r_1\le\cdots\le r_n=1\}$ of the interval $[0, 1]$ is a sequence:
$$(\al_1,g_1,\al_2,g_2,\ldots, \al_n)$$
where 
\begin{enumerate}
\item for all $1\le i\le n$ the map $\al_i:[r_{i-1}, r_i]\to G_0$ is a path with $\al_1(0)=x$ and $\al_n(1)=y$
\item  for all $1\le i\le n-1$ the arrow $g_i\in G_1$ satisfies:

$s(g_i)=\al_i(r_i)$  

$t(g_i)=\al_{i+1}(r_i)$ 
\[
\xy
 (7,0)*{\bullet}="2";
  (21,0)*{\bullet}="3";
   (35,0)*{\bullet}="4";
 (49,0)*{\bullet}="5";
  (63,0)*{\bullet}="6";
    (4,0)*{x};
    (66,0)*{y};
  (14,2.5)*{\scriptstyle \al_1};
   (28,2.5)*{\scriptstyle g_1};
   (56,2.5)*{\scriptstyle \al_n};
   "3"; "4" **\dir{-} ?(.55)*\dir{>};
 { \ar@{~>}"2"; "3"};
 { \ar@{~>}"5"; "6"};
{ \ar@{}"4"; "5"};
(42,0)*{\cdots};
\endxy 
\]
\end{enumerate}

Our notion of  homotopy $H^S: \cK\times \I_S \to \G$ between $H^S_{\bf{0}}= \phi$ and $H^S_{\bf{1}}= \psi$
 determines for each $x\in K_0$ a $\G$-path over the subdivision $S$ between the objects $\phi(x)$ and $\psi(x)$ in $G_0$. In fact, it determines many more paths in $G_1$ other than the one defined by the unit arrows. If we think of a $\G$-path as a morphism from some version of the interval $I$ to $\G$, we have that the Haefliger $\G$-paths correspond to morphisms  $\sigma\colon \I_S\to \G$, where the groupoid $\I_S$ is the one constructed above.

\begin{remark}\label{gpath}  A homotopy $H^S: \cK\times \I_S \to \G$ determines a $\G$-path $H^S_{\bf{1}_x}\colon \I_S \to \G $  where ${\bf{1}_x}$ is the trivial groupoid over $x\in K_0$.
\end{remark}
 
 Note that thinking of a $\G$-path as a morphism $\sigma$ from $\I_S$ to $\G$, the image of the arrows in $\I_S$ by $\sigma$ is almost entirely  contained in $u(G_0)$ since most of the arrows in $\I_S$ are units, with the only exception of the arrows $\{r_1,\cdots, r_{n-1}\}$ and their inverses. In other words, we can not fill the space $G_1$ with paths of arrows given by the $\G$-paths.
 This is going to be a very relevant problem when we try to integrate a vector field. Motivated by these considerations, we propose the notion of {\it multiple} $\G$-path.

To define the several branches coming into a multiple $\G$-path we need to introduce several copies of each subinterval in the subdivision $S$.
 We define a groupoid $\I'_S$ associated to the interval $I$ and the subdivision $S$ in the following way.  Given a subdivision $S=\{0=r_0\le r_1\le\cdots\le r_n=1\}$ of the interval $I=[0,1]$ consider  the  manifold of objects  given by the disjoint union:
$$\bigsqcup_{i=1}^{n} \bigsqcup_{j=1}^{m_i}[r_{i-1}, r_{i}]_j$$
where the extra copies of each interval are indexed by $j$.

\[
\xy
 (7,0)*{[}="2";
  (21,0)*{]}="3";
   (35,0)*{}="4";
 (49,0)*{[}="5";
  (63,0)*{]}="6";
   (7,-4)*{\scriptstyle 0};
    (21,-4)*{\scriptstyle r_1};
  (14,2.5)*{\scriptstyle };
   (28,0)*{\scriptstyle  \sqcup};
    (40,0)*{\scriptstyle  \sqcup};
    (49,-4)*{\scriptstyle r_{n-1}};
    (63,-4)*{\scriptstyle 1};
   (56,2.5)*{\scriptstyle };
    "2"; "3" **\dir{-} ;
    "5"; "6" **\dir{-};
{ \ar@{}"4"; "5"};
(34,0)*{\cdots};
(7,-20)*{[}="12";
  (21,-20)*{]}="13";
   (35,-10)*{}="14";
 (49,-20)*{[}="15";
  (63,-20)*{]}="16";
   (7,-24)*{\scriptstyle 0};
    (21,-24)*{\scriptstyle r_1};
  (14,2.5)*{\scriptstyle };
   (14,-4)*{\scriptstyle  \sqcup};
     (14,-16)*{\scriptstyle  \sqcup};
     (14,-10)*{\vdots};
     (56,-4)*{\scriptstyle  \sqcup};
     (56,-16)*{\scriptstyle  \sqcup};
     (56,-10)*{\vdots};
     (35,-10)*{\vdots};
    (49,-24)*{\scriptstyle r_{n-1}};
    (63,-24)*{\scriptstyle 1};
   (56,2.5)*{\scriptstyle };
    "12"; "13" **\dir{-};
    "15"; "16" **\dir{-};
{ \ar@{}"14"; "15"};
\endxy 
\]
An element in the connected component  $[r_{i-1}, r_{i}]_j$ will be denoted by $(r,i,j)$. Then the manifold of objects  is $${(\I'_S)}_0=\{(r,i,j)| r\in [r_{i-1}, r_{i}], j=1, \ldots, m_i,  i=1, \ldots, n\}.$$
The manifold of arrows of $\I'_S$ is generated by the disjoint union
of:
\begin{enumerate}
\item  $\displaystyle\bigsqcup_{i,j}  [r_{i-1}, r_{i}]_j$  the set of unit arrows;
\item $\displaystyle\bigsqcup_{i} r_i$  the set of arrows connecting the jumps at the subdivision points $r_i$, i.e. the source and target of the arrow $r_i$ are $s(r_i)=(r_i, i,1)$ and  
$t(r_i)=(r_i, i+1,1)$ and 
\item $\displaystyle\bigsqcup_{i,j} [r_{i-1}, r_{i}]_j$ the set of arrows between the different copies $[r_{i-1}, r_{i}]_j$ of each subinterval, i.e. the source and target of the arrow $r_{ij}\in[r_{i-1}, r_{i}]_j$ are $s(r_{ij})=(r,i,j)$ and $t(r_{ij})=(r,i,j+1)$.
\end{enumerate}

\[
\xy
 (7,0)*{[}="0";
  (21,0)*{]}="1";
   (7,-5)*{}="2";
   (2,-10)*{\scriptstyle r_{ij}};
    (8,-5)*{}="20";
    (9,-5)*{}="30";
    (15,-10)*{\cdots};
    (21,-5)*{}="40";
  (21,-5)*{}="3";
   (35,0)*{}="4";
 (28,0)*{[}="5";
  (42,0)*{]}="6";
    "0"; "1" **\dir{-} ;
    "5"; "6" **\dir{-};
{ \ar@{}"4"; "5"};
 "1";"5" **\crv{(24,4) & (25,4)};
  "1";"5" **\crv{(24,-4) & (25,-4)};
  (24.5,3)*{\scriptstyle >}+(0,3)*{\scriptstyle r_i};
    (24.5,-3)*{\scriptstyle <}+(0,-3)*{\scriptstyle r^{-1}_i};
(7,-20)*{[}="10";
  (21,-20)*{]}="11";
  (7,-15)*{}="12";
  (8,-15)*{}="22";
  (9,-15)*{}="32";
   (21,-15)*{}="42";
    "10"; "11" **\dir{-};
{ \ar@{->}"2"; "12"};
{ \ar@{->}"20"; "22"};
{ \ar@{->}"30"; "32"};
{ \ar@{->}"40"; "42"};
\endxy 
\]

 \begin{definition}\label{mpath} A multiple $\G$-path over a subdivision $S$ is a morphism $\sigma\colon\I'_S\to \G$. 
\end{definition}

Note that a $\G$-path in the sense of Haefliger is a multiple $\G$-path over the same subdivision by taking $j=1$ for all subintervals and 
$\sigma((r,i))=\al_i(r)$ on objects; $\sigma(r_i)=g_i$ on arrows $r_i$ and $\sigma(u(r,i))=u(\al_i(r))$ for unit arrows. 

We can think of a multiple $\G$-path as a $\G$-path between {\it orbits} or as a path between orbit subgroupoids. In this spirit, we will say that the initial subgroupoid of the path is $\sigma(\bf{0'})$ and the end subgroupoid is $\sigma(\bf{1'})$. Where $\bf{0'}$ and $\bf{1'}$ are the {\it full} subgroupoids over the orbits of  $0$ and $1$, which in general will not be trivial groupoids.

\begin{example}
Consider the Lie groupoid $\G$ given in example \ref{tear} where $G_0=D_1\sqcup D_2$. We show in  figure \ref{G} the image in $G_0$ of a Haefliger $\G$-path between the points $x$ and $y$ given by the map $\sigma\colon \I_S\to \G$ for the subdivision $S=\{0,{1\over 2}, 1\}$.
\begin{figure}
\includegraphics[height=1in]{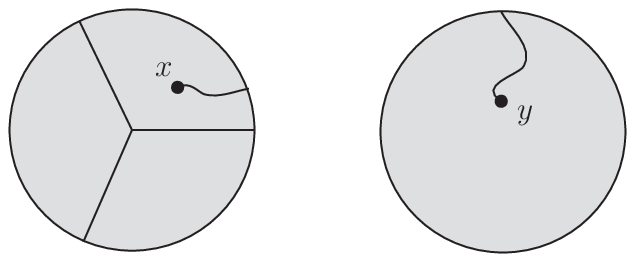}
\caption{}
\label{G}
\end{figure}

In  figure \ref{multi} we see the image of a multiple $\G$-path between the orbits of $x$ and $y$ given by $\tau\colon  \I'_S\to \G$ over the same subdivision. The manifold of objects of $ \I'_S$ is now a disjoint union of three copies of the interval $[0,{1\over 2}]$ plus one copy of the interval $[{1\over 2},1]$.
\begin{figure}
\includegraphics[height=1in]{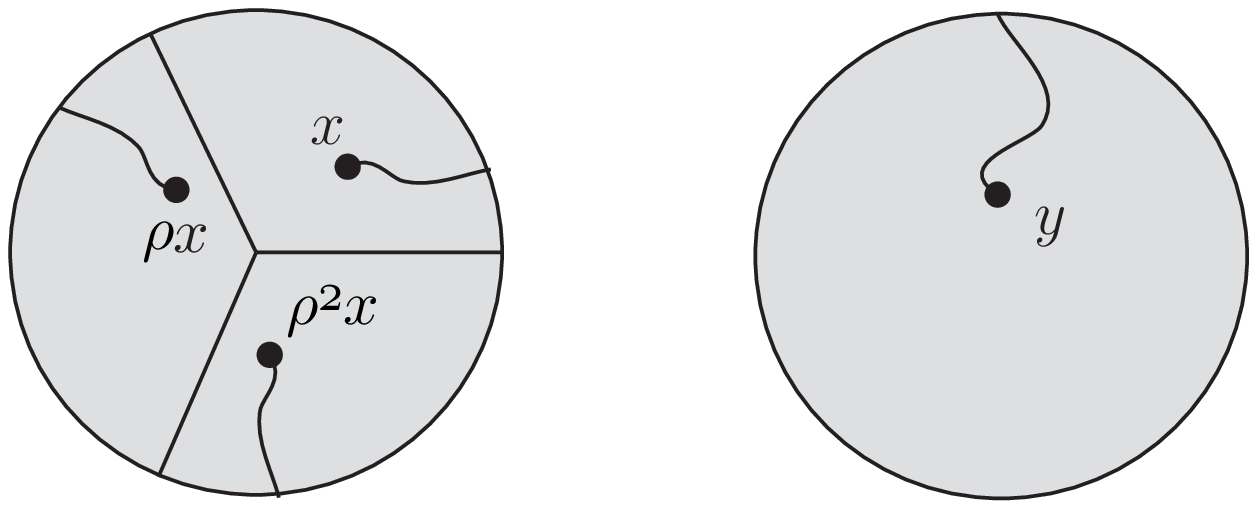}
\caption{}
\label{multi}
\end{figure}

\end{example}

\begin{remark}
A homotopy $H^S: \cK\times \I_S \to \G$, when restricted to the full subgroupoid ${\bf{X}}$ over an orbit $\cO(x)$, defines a multiple $\G$-path between the orbit subgroupoids $H({\bf{X}},\bf{0})$ and $H({\bf{X}},\bf{1})$. If $H^S$ is a homotopy between $\phi$ and $\psi$, then for all $x\in K_0$, the $S$-homotopy defines a multiple $\G$-path between the orbits of $\phi(x)$ and $\psi(x)$.
\end{remark}

Sometimes we will denote a multiple $\G$-path by $$\sigma=(\al^{j_1}_1,g_1, \al^{j_2}_2 \ldots, \al^{j_{n-1}}_{n-1}, g_{n-1}, \al^{j_n}_n)$$ where 
$(\al^{1}_1,g_1, \al^{1}_2 \ldots, \al^{1}_{n-1}, g_{n-1}, \al^{1}_n)$ is a $\G$-path and $$(\al^{1}_1, g_1, \al^{1}_2, \ldots, h_{ij} g_{i-1} ,\al^{j}_{i}, g_{i}h^{-1}_{ij}, \ldots,
 \al^{1}_{n-1}, g_{n-1}, \al^{1}_n)$$ is a $\G$-path for each $j=1, \ldots, m_i$ and each $i=1, \ldots, n$.

\subsubsection{Essential homotopy equivalences} 

\begin{definition}
A morphism $\eta: \cK \to \G$ of groupoids is an {\it essential homotopy equivalence} of groupoids if
there exists an $S$-homotopy equivalence $h\colon \cK\to \cL$ and an essential equivalence $\e\colon \cL \to \G$ such that $\eta\simeq_S \e h$.
$$
\xymatrix@!=1pc{
{\cK}\ar[rr]^{\eta}\ar[rd]_{h}&&{\G}\\
&{\cL}\ar[ur]_{\e}&
}
$$

\end{definition}
This implies that for any object $y\in G_0$, there exists an object  $x\in K_0$ whose image $\eta(x)$ can be connected to $y$ by a concatenation of paths and arrows. Also $\eta$ induces a homotopy equivalence  $|\eta|$ between the orbit spaces.

\begin{remark} \label{isotropy} If  $\eta: \cK \to \G$  is an essential homotopy equivalence, then there is an injection between the corresponding isotropy groups $K_x$ and $G_{\eta(x)}$.
\end{remark}

\begin{example} \label{ex} Consider the inclusion functor  $\eta: \bullet^{\Z_2} \to \Hol(M,\F_S)$ where $ \F_S$ is the Seifert fibration of the M\"obius band $M$. Let $x=\eta(\bullet)$ be a fix point in the central fiber in $M$. Consider the inclusion functors
$h\colon \bullet^{\Z_2}\to {\Z_2}\ltimes I $  and $\e\colon {\Z_2}\ltimes I \to \Hol(M,\F_S)$. The functor $h$ is a homotopy equivalence with homotopic inverse the constant map and $\e$ is an essential equivalence. We have that 
$\eta= \e h$  is an essential homotopy equivalence. Note that in this case $\eta$ is neither a (strong) homotopy equivalence  nor an essential equivalence.
\end{example}

It is clear that essential equivalences as well as (strong) homotopy equivalences are essential homotopy equivalences.

\begin{prop}\label{BF1} 
\begin{enumerate}
\item An essential equivalence is  an essential homotopy equivalence. 
 \item A homotopy equivalence is an essential homotopy equivalence.
 \end{enumerate}
\end{prop}

\begin{prop}\label{BF4} If $\eta\phi \simeq_S \eta\psi$ and $\eta$ is an essential homotopy equivalence, then $\phi \simeq_S \psi$.
\end{prop}
\begin{proof}
Since $\eta \simeq_S \e h$, we have that $\e h \phi \simeq_S \e h \psi$. Then $ h \phi \simeq_S  h \psi$ because $\e$ is an essential equivalence. If $g$ is the homotopic inverse of $h$, we have that $ g h \phi \simeq_S g h \psi$. Thus, $  \phi \simeq_S  \psi$.
\end{proof}

Let $\phi: \cK \to \G$ and   $\psi: \cJ \to \G$ be morphisms of Lie groupoids and $S$ a subdivision of the interval $I=[0,1]$. Let $P_S(\G)$ be the space of  $\G$-paths over the subdivision $S$.
\begin{definition}
The {\it groupoid homotopy pullback }  $\cP_S=\cK\times_S^h \cJ$ is the Lie groupoid whose manifold of objects is 
$$(\cP_S)_0=\{(x,\sigma, y)| x\in K_0, y\in J_0, \sigma\in P_S(\G)  \mbox { with } \sigma(0)=\phi(x) \mbox { and } \sigma(1)=\psi(y)\}$$ and whose manifold of arrows is 
$$(\cP_S)_1=\{(k,\sigma, j)| k\in K_1, j\in J_1,  \sigma\in P_S(\G)  \mbox { with } \sigma(0)=\phi(s(k)) \mbox { and } \sigma(1)=\psi(t(j))\}$$
and source and target maps are given by:  $s(k,\sigma, j)=(s(k),\sigma, s(j))$ and $t(k,\sigma, j)=(t(k),\phi(j)\sigma\phi(k)^{-1}, t(j))$.
\end{definition}

The groupoid homotopy pullback is well defined (up to homotopy) whenever one of the maps $\phi$ or $\psi$ is homotopic to a submersion on objects. 

We will show the construction of the manifold of objects $(\cP_S)_0$ for a subdivision $S=\{t_0=0,t_1, t_2,  t_3=1\}$ by a sequence of pullbacks and homotopy pullbacks of manifolds. 

We recall first the definition of homotopy pullback of manifolds \cite{Ma}. Given two maps $f: X \to Z$ and   $g: Y \to Z$, the {\it homotopy pullback} is the space 
$$X\times^hY=\{(x,\al, y)| x\in X, y\in Y \mbox { and } \al  \mbox { is a path between } f(x) \mbox { and } g(y)\}.$$
The following diagram commutes up to homotopy and it is universal (up to homotopy) with this property.

\[\xymatrix{
X\times^hY \ar[r]^{p_1} \ar[d]_{p_2}& X \ar[d]^{f} \\ 
Y\ar[r]^{g} & Z}\]

\begin{remark} \label{manifold} If $f$ is homotopic to a submersion, then the homotopic pullback $X\times^hY $ is a manifold homotopic to the ordinary pullback.
\end{remark}

Consider the following {\it ordinary} pullbacks of manifolds:
$$\xymatrix{ K_0\times_{G_0}G_1
\ar[r]^{p_1} \ar[d]_{p_2}& K_0 \ar[d]^{\phi} \\ 
G_1\ar[r]^{s} & G_0} \mbox{ and }
\xymatrix{ J_0\times_{G_0}G_1
\ar[r]^{p'_1} \ar[d]_{p'_2}& J_0 \ar[d]^{\psi} \\ 
G_1\ar[r]^{t} & G_0}$$
and the following {\it homotopy} pullback of manifolds:
$$\xymatrix{ (\cP_S)_0
\ar[rr]^{} \ar[dd]_{p}& &K_0\times_{G_0}G_1 \ar[d]^{p_2} \\ 
&&G_1\ar[d]^{t}\\
J_0\times_{G_0}G_1\ar[r]^{p'_2} & G_1\ar[r]^{s} &G_0} $$
We have that  $(\cP_S)_0$ is the manifold $$\{(x,g,\alpha, h, y)|  \al  \mbox { is a path between } t(g)  \mbox { and } s(h) \mbox { with }  \phi(x)=s(g) \mbox { and }\psi(y)=t(h)\}.$$
Analogously, we can define the manifold of arrows $(\cP_S)_1$ as the homotopy pullback $$(K_1\times_{G_0}G_1)\times^h (J_1\times_{G_0}G_1)=K_1\times_{G_0}G_1  \times^h G_1\times_{G_0}J_1.$$

If $(k,g,\alpha, h, j)$ is an arrow from $(x,g,\alpha, h, y)$ to $(x',g\phi(k)^{-1},\alpha, \psi(j)h, y')$ and $(k',g',\alpha, h', j')$ is an arrow from $(x',g',\alpha, h', y')$ to $(x'',g'\phi(k')^{-1},\alpha, \psi(j')h', y'')$,
the composition of arrows is given by 
$$(k',g,\alpha, h, j')\circ (k,g',\alpha, h', j)=(k'k,g,\alpha, h, j'j)$$

$$\xy
 (13,0)*{\bullet}="1";
 (27,0)*{\bullet}="2";
  (41,0)*{\bullet}="3";
   (55,0)*{\bullet}="4";
   (13,-14)*{\bullet}="7";
   (55,-14)*{\bullet}="8";
   (13,-28)*{\bullet}="9";
   (55,-28)*{\bullet}="12";
   (20,2.5)*{\scriptstyle g};
  (34,2.5)*{{\scriptstyle \al}};
   (48,2.5)*{{\scriptstyle h}};
   (20,-11.5)*{\scriptstyle g'};
   (48,-11.5)*{{\scriptstyle h'}};
   (20,-25.5)*{\scriptstyle g''};
   (48,-25.5)*{{\scriptstyle h''}};
  (8,-7)*{\scriptstyle \phi(k)};
   (60,-7)*{\scriptstyle \psi(j)};
   (8,-21)*{\scriptstyle \phi(k')};
   (60,-21)*{\scriptstyle \psi(j')};
    "1"; "2" **\dir{-} ?(.55)*\dir{>};
   "3"; "4" **\dir{-} ?(.55)*\dir{>};
    "1"; "7" **\dir{-} ?(.55)*\dir{>};
     "4"; "8" **\dir{-} ?(.55)*\dir{>};
      "7"; "9" **\dir{-} ?(.55)*\dir{>};
       "8"; "12" **\dir{-} ?(.55)*\dir{>};
       "7"; "2" **\dir{-} ?(.55)*\dir{>};
     "3"; "8" **\dir{-} ?(.55)*\dir{>};
      "9"; "2" **\dir{-} ?(.55)*\dir{>};
       "3"; "12" **\dir{-} ?(.55)*\dir{>};
   { \ar@{~>}"2"; "3"};
\endxy $$

We can generalize the construction to obtain the groupoid homotopy pullback $\cP_S$ corresponding to the subdivision $S=\{0=r_0\le r_1\le\cdots\le r_n=1\}$
by iterating $n$ homotopy pullbacks. Consider the homotopy pullback $K_0\times^hG^{n-1}=K_0\times^h\overbrace{ G_1\times^h\cdots  \times^h G_1 }^{n-1}$ of length $n$ given by the following diagram: 
\[\xymatrix{
K_0\times^hG^{n-1}\ar[r]^{} \ar[ddddd]_{}& K_0\times^hG^{n-2}_1 \ar[dddd]^{}\ar[r]^{}&\cdots\ar[r]^{}&K_0\times^hG_1\ar[d]^{}\ar[r]^{}&K_0\ar[d]^{} \\ 
& &&G_1\ar[d]^{t}\ar[r]^{s}&G_0\\
&&{\:\:\: \:\:\:  }\ar[r]^{s}&G_0&\\
&&\Ddots&&\\
&G_1\ar[d]^{t}\ar[r]^{s}&{\:\:\: \:\:\:  }&&\\
G_1\ar[r]^{s}&G_0&&&
}\]
Then, we have that

$$(\cP_S)_0= K_0\times_{G_0}^h G^{n-1}\times_{G_0}^h J_0$$ and 
$$(\cP_S)_1= K_1\times_{G_0}^h G^{n-1}\times_{G_0}^h J_1.$$

We will show that the following diagram of groupoids commutes up to strong homotopy:
\[\xymatrix{ \cP_S
 \ar[r]^{\pi_1} \ar[d]_{\pi_3}& \cK \ar[d]^{\phi} \\ 
\cJ\ar[r]^{\psi} & \G}\]
Consider the homotopy $H^S: \cP_S\times \I_S \to \G$ given by $H^S((x,\sigma, y), (r,i,j))=\sigma((r,i,j))$ on objects and $H^S((k,\sigma, j), r_{i})=\sigma(r_{i})$ on arrows (see remark \ref{gpath}). We have that  $H^S_{\bf{0}}= \phi\pi_1$ and $H^S_{\bf{1}}= \psi\pi_3.$

 The homotopy pullback of groupoids satisfies the following universal property. For any groupoid $\cA$ and morphisms $\delta:\cA\to \cK$ and  $\gamma:\cA\to \cJ$ with $\phi \delta\simeq_{H^S} \psi \gamma$ there exists a morphism  $\xi\colon \cA\to \cP_S$ such that both triangles commute up to homotopy:
 $$ \xymatrix{
\cA \ar@/_/[ddr]_{\gamma} \ar@/^/[drr]^{\delta}
\ar@{.>}[dr]|-{\xi} \\
& \cP_S \ar[d]^{\pi_3} \ar[r]_{\pi_1}
& \cK \ar[d]_{\phi}\\
& \cJ \ar[r]^{\psi} & \G }$$
We define  $\xi\colon \cA\to \cP_S$ by $\xi(w)=(\delta(w), H^S_{\bf{1}_w}, \gamma(w))$ on objects and $\xi(a)=(\delta(a), H^S_{\bf{1}_w}, \gamma(a))$ on arrows, where ${\bf{1}_w}$ is the trivial groupoid over $w$.

Note that the groupoid homotopy pullback $\cP_{S}$ is defined for each subdivision $S$ of the interval $I=[0,1]$.

\begin{remark} If $S=\{0=r_0\le r_1\le\cdots\le r_n=1\}$ is a subdivision of $I=[0,1]$, then for all subdivision $S'\supset S$ we have:
\begin{enumerate} 
\item If $\phi\simeq_{H^S}\psi$, then $\phi\simeq_{H^{S'}}\psi$.
\item There exists a morphism $\xi\colon \cP_{S}\to \cP_S'$ such that $\pi'_1\xi=\pi_1$ and $\pi'_3\xi=\pi_3$. Since $\phi\pi_1\simeq_{H^S}\psi\pi_3$, then $\phi\pi_1\simeq_{H^{S'}}\psi\pi_3$ and the following diagram commutes up to $S'$-homotopy:
$$ \xymatrix{
\cP_S  \ar@/_/[ddr]_{\pi_3} \ar@/^/[drr]^{\pi_1}
\ar@{.>}[dr]|-{\xi} \\
& \ar@{}[dr] |{\Downarrow_{S'}}\cP_{S'} \ar[d]_{\pi'_3} \ar[r]^{\pi'_1}
& \cK \ar[d]_{\phi}\\
& \cJ \ar[r]^{\psi}& \G }$$
\end{enumerate}\end{remark}

\begin{prop}\label{order}
If there is a subdivision $S$ such that the following diagram commutes up to $S$-homotopy:
$$ \xymatrix{
 \ar@{}[dr]^{\simeq_{S}}\cK \ar[dr]_{\e} \ar[r]^{\eta} &\G \\
 & \cL \ar[u]_{h}}$$
 where $\e$ is an essential equivalence and $h$ is a homotopy equivalence, then
$\eta: \cK \to \G$ is an essential homotopy equivalence. 
\end{prop}
\begin{proof} Let $g$ be the homotopic inverse of $h$. Then, the following triangle commutes up to $S$-homotopy:
 
$$ \xymatrix{
 \G\ar@{}[dr]_{\simeq_{S}} \ar[dr]^{\id} \ar[d]_{g} &\\
\cL \ar[r]_{h}& \G }$$ 

Consider the following fibered product of groupoids:

$$ \xymatrix{
 \ar@{}[dr] |{\sim}\G\times_{\cL}\cK \ar[d]_{g'} \ar[r]^{\e'}&\G \ar[d]^{g}\\
 \cK \ar[r]_{\e}& \cL }$$
 
 The square above commutes up to a natural transformation, then it commutes up to $S$-homotopy and has the universal property. Since $\eta\simeq_S h\e$, we have that the following diagram commutes up to $S$-homotopy:
 
 $$ \xymatrix{
\cK \ar@/_/[ddr]_{\id} \ar@/^/[drr]^{\eta}\ar@{.>}[dr]|-{\xi} \\
& \ar@{}[dr] |{}\G\times_{\cL}\cK \ar[d]_{g'} \ar[r]^{\e'}&\G \ar[d]^{g}\ar[dr]^{\id}&\\
 &\cK \ar[r]_{\e}& \cL  \ar[r]_{h}&\G}$$
 Therefore, $\e'\xi\simeq_{S}\eta$ and $g'\xi\simeq_{S}\id$. We have that: 
 $$\eta g'\simeq_Sh \e g' \simeq_Sh g \e' \simeq_S \id \e' \simeq_S \e'$$ and also:
 $$\eta g' \simeq_S \eta \id_{\cK} g' \simeq_S \eta g' \xi g' \simeq_S \e' \xi g'.$$
 By proposition \ref{BF4}, we have that $\id\simeq_{S}\xi g'$ and $\xi$ is the homotopic inverse of $g'$. Then $\eta\simeq_{S}\e' h'$ with $\e'$ and essential equivalence and $h'=\xi$ a homotopy equivalence.
\end{proof}

\begin{lemma}\label{above}
If $\e$ is an essential equivalence and $g$ is an $S$-homotopy equivalence, then the homotopic pullback 
\[\xymatrix{ \cP_S
 \ar[r]^{\e'} \ar[d]_{g'}& \cJ \ar[d]^{g} \\ 
\cL\ar[r]^{\e} & \G}\]
exists and $g'$ is an $S$-homotopy equivalence as well.
\end{lemma}
\begin{proof}
Since $\e$ is an essential equivalence, the pullback  $\cP_S$ exists. Let $f\colon \G \to \cJ$ be the homotopic inverse of $g$, then $fg \simeq_S \id_{\cJ}$ and $gf \simeq_S \id_{\G}$. Consider the following pullbacks:

\[\xymatrix{ 
\cP'_S \ar[r]^{\e''} \ar[d]_{f'}& \G \ar[d]^{f} \\ 
\cP_S \ar[r]^{\e'} \ar[d]_{g'}& \cJ \ar[d]^{g} \\ 
\cL\ar[r]^{\e} & \G}\]
By the universal property of the large square, there exists $\xi\colon \cL \to \cP'_S$ such that the following diagram commutes up to $S$-homotopy:

 $$ \xymatrix{
\cL \ar@/_/[ddr]_{\id} \ar@/^/[drr]^{\e}\ar@{.>}[dr]|-{\xi} \\
& \ar@{}[dr] |{}\cP'_S \ar[d]^{g'f'} \ar[r]^{\e''}&\G \ar[d]^{\id}\\
 &\cL \ar[r]_{\e}& \G  }$$
Then $g'f'\xi  \simeq_S \id$ and $\e''\xi  \simeq_S \e$. We will see that $f'\xi$ is the homotopic inverse of $g'$. We have that
$$\e' f'\xi g'  \simeq_S f\e''\xi g' \simeq_S f\e g'  \simeq_S fg\e'  \simeq_S \e'.$$
By proposition \ref{BF4}, $f'\xi g'  \simeq_S \id$ and $g'$ is a homotopy equivalence.
\end{proof}

\begin{prop} \label{BF2}
If $\eta$ and $\nu$ are essential homotopy equivalences, then $\nu\eta$ is an essential homotopy equivalence as well.
\end{prop}
\begin{proof}
We have that $\eta \simeq_{S'} \e h$ and $\nu \simeq_{S''} \sigma f$ with $\e$ and $\sigma$ being essential equivalences and $h$ and $f$ homotopy equivalences. Consider the following homotopic pullback, where $S$ is a refinement of $S$ and $S'$:

\[\xymatrix{ \cP_S
 \ar[r]^{g'} \ar[d]_{\e'}& \cA \ar[d]^{\e} \\ 
\cB\ar[r]^{g} & \G}\]
where $g$ is the homotopic inverse of $f$. By lemma \ref{above}, we have that $g'$ is a homotopy equivalence. Then, the following diagram commutes up to $S$-homotopy:

\[
\xymatrix@!=1pc{\cK\ar[dr]_{h}\ar[rr]^{\eta}&&\G\ar[rr]^{\nu}&&\cL\\
&\cA\ar[ur]_{\e}&&\cB\ar[ur]_{\sigma}\ar[ul]^{g}&\\
&&\cP_S \ar[ul]^{g'}\ar[ur]_{\e'}&&
}
\]
Let $h'$ be the homotopic inverse of $g'$. Then, $\nu\eta \simeq_{S} \sigma \e' h' h$ where $\sigma \e'$ is an essential equivalence and $h h'$ is a homotopy equivalence.
\end{proof}
We can pull back the decomposition given in proposition \ref{order} and we have the following

\begin{prop} \label{BF3}
If $\nu: \cK \to \G$ is an essential homotopy equivalence, then there is a subdivision $S$ such that the morphism $\pi_3: \cP_S \to \cJ$ is an essential homotopy equivalence as well.
\end{prop}

\begin{definition}\label{mhe}
A Lie groupoid $\cK$ is {\it Morita homotopy equivalent} to $\G$ if there exists a Lie groupoid $\cL$ and essential homotopy equivalences
\[\cK\overset{\eta}{\gets}\cL\overset{\nu}{\to}\G.\]
\end{definition}

Note that it is possible to choose these essential homotopy equivalences being homotopic to surjective submersions on objects.

By using two homotopy pullbacks of groupoids and proposition \ref{BF3}, it is easy to see that Morita homotopy equivalence is a relation of equivalence.

\section{The Morita homotopy bicategory of Lie groupoids}
We will introduce in this section a bicategory whose objects are Lie groupoids and where Morita homotopy equivalence amounts to equivalence of objects. We follow a bicategorical approach as emphasized by Landsman in \cite{La1,La}.
 \subsection{Bicategories of fractions} 
Recall that a bicategory $\BB$ consists of a class of objects, morphisms between objects and 2-morphisms between morphisms together with various ways of composing them. 
 A 2-morphism is a {\it 2-equivalence }  if it is invertible in $\BB$ whereas a morphism  is an {\it equivalence } if it is invertible up to a 2-equivalence in $\BB$. We have that 
 $H\colon \phi\st{\sim}\rr\psi$ is a 2-equivalence in $\BB$ if there exists $G$ such that $HG=\id$ and $GH=\id$, and $\phi\colon \cK\st{\sim}\r\G$ is an equivalence in $\BB$ if there exists $\psi$ such that $\phi\psi\st{\sim}\rr\id$ and $\psi\phi\st{\sim}\rr\id$.

 We can compose 2-morphisms in two ways called horizontal and vertical composition.Vertical composition is strictly associative whereas horizontal composition is only associative up to 2-equivalence (associator). The unit laws for morphisms hold up to 2-equivalences (left and right unit constrains). Associator and unit constrains are required to be natural with respect to their arguments and verify certain axioms \cite{Be}.

 Given a bicategory $\BB$ and a subset $A\subset B_1$ of morphisms satisfying certain conditions, there exists a bicategory $\BB(A^{-1})$ having the same objects as $\BB$ but inverse morphisms of morphisms in $A$ have been added as well as more 2-morphisms. This new bicategory is called a {\it bicategory of fractions} of $\BB$ with respect to $A$ and was constructed by Pronk in \cite{P}. The bicategory of fractions $\BB(A^{-1})$ is characterized by the universal property that any morphism $F:\BB\to\DD$ sending elements of $A$ into equivalences factors in a unique way as $F=\tilde F\circ U$
$$         
\xymatrix{
\BB \ar^{F}[d] \ar^{U\;\;\;\;}[r] & \BB(A^{-1})\ar^{\tilde F}[ld] \\ 
\DD & } 
$$
The following conditions are needed on $A$ to admit a bicalculus of right fractions  \cite{P}:
\begin{enumerate}
\item[BF1] All equivalences are in $A$.
\item[BF2] If $\phi$ and $\psi$ are in $A$, then $\phi\psi\in A$.
\item[BF3] For all $\eta:\cJ\to\G$ and $\phi:\cK\to\G$ with $\eta\in A$ there exists an object $\cP$ and morphisms $\nu:\cP\to\cK$ and $\psi:\cP\to\cJ$ with $\nu\in A$ such that the following square commutes up to a 2-equivalence:
$$         
\xymatrix{
\cP \ar^{\nu}[d] \ar^{\psi}[r] & \cJ\ar^{\eta}[d] \\ 
\cK \ar^{\phi}[r] & \G } 
$$
\item[BF4] If $H:\eta\phi\ri\eta\varphi$ is a 2-morphism with $\eta\in A$ then there exists a morphism $\nu\in A$ and a 2-morphism $F:\phi\nu\ri\varphi\nu$ such that $H\cdot\nu=\eta\cdot F$.
\item[BF5] If there is a 2-equivalence $\eta\st{\sim}\rr\nu$ with $\nu\in A$ then $\eta\in A$.
\end{enumerate}

\begin{prop} The set of Lie groupoids, morphisms and (homotopy classes of) strong homotopies as 2-morphisms forms a bicategory $\HH$. 
\end{prop}
\begin{proof}
Strong homotopies $H^S\colon \cK\times\I_S\to\G$ and $H'^{S'}\colon \cK\times\I_{S'}\to\G$ can be composed vertically by taking the subdivision $S''=\{0=r_0\le {r_1\over 2}\le\cdots\le  {r_n\over 2}, {r'_0+1\over 2},  \le\cdots\le {r'_m+1\over 2}=1\}$. We take homotopy classes of homotopies to assure associativity of the vertical composition.
\end{proof}
We will formally invert the set of essential homotopy equivalences $W$. 
\begin{prop}
The set  $W$ of essential homotopy equivalences allows a bicalculus of right fractions.
\end{prop}
\begin{proof}
Condition BF1 follows from proposition \ref{BF1}. Condition BF2 follows from proposition \ref{BF2}. Condition BF3 follows from proposition \ref{BF3}. Condition BF4 follows from proposition \ref{BF4}.
Finally, if $\eta\ri\nu$ is a strong homotopy equivalence and $\eta$ is an essential homotopy equivalence, then $\nu \simeq_S \eta \simeq_S h \e$ and $\nu$ is an essential homotopy equivalence as well.

Therefore, there exists a bicategory of fractions $\HH(W^{-1})$ inverting the essential homotopy equivalences.
\end{proof}

\subsection{The bicategory $\HH(W^{-1})$} The objects of $\HH(W^{-1})$ are Lie groupoids. The 1-morphisms from $\cK$ to $\G$ are formed by pairs $(\eta,\phi)$
\[\cK\overset{\eta}{\gets}\cJ\overset{\phi}{\to}\G\]
such that $\eta$ is an essential homotopy equivalence. The 1-morphisms $(\eta,\phi)$ in this bicategory will be called {\it generalized homotopy maps}.
If $\e$ is an essential equivalence, then $(\e,\phi)$ is called a generalized map.

A 2-morphism from $(\eta,\phi)$ to $(\eta',\phi')$ is given by  the following diagram:
$$
\xymatrix{ & 
{\cJ}\ar[dr]^{\phi}="0" \ar[dl]_{\eta }="2"&\\
{\cK}&{\cL} \ar[u]_{u} \ar[d]^{v}&{\G}\\
&{\cJ'}\ar[ul]^{\eta'}="3" \ar[ur]_{\phi'}="1"&
\ar@{}"0";"1"|(.4){\,}="7" 
\ar@{}"0";"1"|(.6){\,}="8" 
\ar@{=>}"7" ;"8"_{H^S} 
\ar@{}"2";"3"|(.4){\,}="5" 
\ar@{}"2";"3"|(.6){\,}="6" 
\ar@{=>}"5" ;"6"^{F^S} 
}
$$
where $\cL$ is a Lie groupoid, $u$ and $v$ are essential homotopy equivalences and $F^S:\eta u\ri \eta ' v$ and $H^S:\phi u\ri \phi'v$ are strong homotopy equivalences.

The notion of {\it homotopy} we propose corresponds to 2-morphisms in this bicategory $\HH(W^{-1})$.
\begin{definition} Two generalized homotopy maps $(\eta,\phi)$ and $(\eta',\phi')$ are {\it Morita homotopic} if there is a 2-morphism between them.
\end{definition}
In this case, we write $(\eta,\phi)\simeq(\eta',\phi')$ and we say that there is a {\it Morita homotopy} between $(\eta,\phi)$ and $(\eta',\phi')$.

In particular, when $\eta$ and $\eta'$ are essential equivalences, we have a notion of homotopy for generalized maps and when they are identities, we have a notion of homotopy for strict maps.

Two objects $\cK$ and $\G$  are equivalent in $\HH(W^{-1})$  if there are morphisms $(\eta,\phi)$ from  $\cK$ to $\G$ and $(\theta,\psi)$ from $\G$ to $\cK$ such that $(\eta,\phi)\circ(\theta,\psi)$ is Morita homotopic to the identity $(\id_{\G},\id_{\G})$ and $(\theta,\psi)\circ(\eta,\phi)\simeq(\id_{\cK},\id_{\cK})$.

\begin{prop} A morphism $(\eta,\phi)$ is an equivalence in $\HH(W^{-1})$ if and only if $\phi$ is an essential homotopy equivalence. In this case, the inverse of $(\eta,\phi)$ is the morphism $(\phi,\eta)$.
\end{prop}
 In other words, the definition of  Morita homotopy equivalence in subsection \ref{mhe} amounts to equivalence of objects in the bicategory $\HH(W^{-1})$. So, we write $\cK\simeq\G$ for equivalence of objects in $\HH(W^{-1})$. The {\it Morita homotopy type}  of $\G$ is the class of $\G$ under the equivalence relation $\simeq$.
 
 \begin{example}\label{mht}
 The holonomy groupoid $\G=\Hol(M,\F_S)$ of the Seifert fibration on the M\"obius band has the same Morita homotopy type than the point groupoid $\cK=\bullet^{\Z_2}$ since we have
  $$\bullet^{\Z_2}\overset{\eta}{\gets}\Z_2\ltimes I\overset{\nu}{\to}\Hol(M,\F_S)$$
  where $\eta$ and $\nu$ are essential homotopy equivalences. The homotopic inverse is the generalized map 
  $\Hol(M,\F_S)\overset{\nu}{\gets}\Z_2\ltimes I\overset{\eta}{\to}\bullet^{\Z_2}$. Note that $\cJ=\Z_2\ltimes I$ is Morita equivalent to $\G=\Hol(M,\F_S)$ since the inclusion $\nu=i_{\cJ}$ is an essential equivalence, but the groupoids $\cK$ and $\G$ are not Morita equivalent.
 \end{example}

We show now that the Morita homotopy type is invariant under Morita equivalence.
\begin{prop} If $\cK\sim_M\G$, then $\cK\simeq\G$.
\end{prop}
\begin{proof}
If $\cK$ and $\G$ are Morita equivalent, then there is Lie groupoid $\cL$ and essential equivalences
\[\cK\overset{\eta}{\gets}\cL\overset{\nu}{\to}\G.\]
The morphisms $\eta$ and $\nu$ are also essential homotopy equivalences (see proposition \ref{BF1}). Then the morphisms $(\eta,\nu)$ and $(\nu,\eta)$ are inverse up to equivalence  and $\cK$ is equivalent to $\G$ in the bicategory $\HH(W^{-1})$.
\end{proof}

\begin{remark}
A Morita homotopy equivalence $\cK\simeq\G$ induces a Morita equivalence between the homotopy groupoids $\cK^*\sim_M\G^*$. If $(\eta, \nu)$ is a Morita equivalence, then $(\eta, \nu)$ is a $1$-homotopy equivalence in the sense of \cite{GH}.
\end{remark}

We will give an explicit description of the horizontal and vertical composition in this bicategory. The horizontal composition of 1-morphisms $\cK\overset{\eta}{\gets}\cJ\overset{\phi}{\to}\G$ and $\G\overset{\nu}{\gets}\cJ'\overset{\psi}{\to}\cL$ is given by the diagram

\[
\xymatrix@!=1pc{&&\cP_S\ar[dr]\ar[dl]&&\\
&\cJ\ar[dl]_{\eta}\ar[dr]^{\phi}\rrtwocell<\omit>{<0>S}&&\cJ'\ar[dl]_{\nu}\ar[dr]^{\psi}&\\
\cK &&\G&&\cL
}
\]
and the vertical composition of the 2-arrows 

$$
\xymatrix{ & 
{\cC}\ar[dr]^{\phi}="0" \ar[dl]_{\eta }="2"&\\
{\cK}&{\cL} \ar[u]_{u} \ar[d]^{v}&{\G}\\
&{\cD}\ar[ul]^{\eta'}="3" \ar[ur]_{\phi'}="1"&
\ar@{}"0";"1"|(.4){\,}="7" 
\ar@{}"0";"1"|(.6){\,}="8" 
\ar@{=>}"7" ;"8"_{S} 
\ar@{}"2";"3"|(.4){\,}="5" 
\ar@{}"2";"3"|(.6){\,}="6" 
\ar@{=>}"5" ;"6"^{S} 
}
\qquad
\xymatrix{ & 
{\cD}\ar[dr]^{\phi'}="0" \ar[dl]_{\eta' }="2"&\\
{\cK}&{\cL'} \ar[u]_{u'} \ar[d]^{v'}&{\G}\\
&{\cE}\ar[ul]^{\eta''}="3" \ar[ur]_{\phi''}="1"&
\ar@{}"0";"1"|(.4){\,}="7" 
\ar@{}"0";"1"|(.6){\,}="8" 
\ar@{=>}"7" ;"8"_{S'} 
\ar@{}"2";"3"|(.4){\,}="5" 
\ar@{}"2";"3"|(.6){\,}="6" 
\ar@{=>}"5" ;"6"^{S'} 
}
$$
 is given by the diagram
 
 \[
\xymatrix{ 
&&& {\cC}\ar[ddrrr]^{\phi}="2"  \ar[ddlll]_{\eta}="0" &&&\\
&&&\cL \ar[dr]^{v}\ar[u]_{u}&&&\\
\cK&& \cP_{S''}\ar[dr]^{} \ar[ur]_{}\rrtwocell<\omit>{<0>{\;\;S''}}&&\cD&&\G\\
&&&\cL'\ar[d]^{v'} \ar[ur]_{u'} &&&\\
&&&\cE\ar[uulll]^{\eta''}="1"  \ar[uurrr]_{\phi''}="3" &&&&
\ar@{}"2";"3"|(.4){\,}="5" 
\ar@{}"2";"3"|(.6){\,}="6" 
\ar@{=>}"5" ;"6"^{S''} 
\ar@{}"0";"1"|(.4){\,}="7" 
\ar@{}"0";"1"|(.6){\,}="8" 
\ar@{=>}"7" ;"8"_{S''} 
} 
\] 
where $S''=S\cup S'$.
\begin{remark}    If  the morphisms $(\e,\phi)$ and $(\delta,\psi)$ are generalized maps, the composition given by the fibered product groupoid $\cK\overset{}{\gets}\cJ'\times_{\G} \cJ\overset{}{\to}\cL$ is {\it Morita homotopic} to this horizontal composition in the bicategory $\HH(W^{-1})$. Since $\phi p_1\sim \psi p_3$ we have that $\phi p_1\simeq_S \psi p_3$ and by the universal property of the homotopy pullback, there exists a map $\xi: \cJ'\times_{\G}\cJ\to \cP_S$ such that the following diagram commutes up to homotopy:

$$
\xymatrix@!=2.3pc{
\cJ'\times_{\G} \cJ\ar@/_/[ddr]_{p_3} \ar@/^/[drr]^{p_1}\ar@{.>}[dr]|-{\xi}&&\\
& \cP_S\ar[d]^{\pi_3}\ar [r]^{\pi_1}&\cJ' \ar[d]_{\nu}\\
&\cJ\ar[r]^{\phi}&\G
}
$$ 
therefore there is a 2-morphism between the generalized homotopy maps $\cK\overset{}{\gets}\cJ'\times_{\G} \cJ\overset{}{\to}\cL$  and $\cK\overset{}{\gets}\cP_S\overset{}{\to}\cL$  given by the following diagram:
$$
\xymatrix{& & 
{\cJ'\times_{\G} \cJ}\ar[dd]^{\xi}\ar[dr]^{p_1}="0" \ar[dl]_{p_3 }="2"&&\\
{\cK}&{\cJ} \ar[l]_{\eta} &&\cJ'\ar[r]^{\psi}&\cL\\
&&\cP_S\ar[ul]^{\pi_3}="3" \ar[ur]_{\pi_1}="1"&
\ar@{}"0";"1"|(.4){\,}="7" 
\ar@{}"0";"1"|(.6){\,}="8" 
\ar@{=>}"7" ;"8"_{} 
\ar@{}"2";"3"|(.4){\,}="5" 
\ar@{}"2";"3"|(.6){\,}="6" 
\ar@{=>}"5" ;"6"^{} 
}$$
Then, when the generalized homotopy maps are generalized maps, we can set the groupoid fibered product as our chosen homotopy pullback.
\end{remark}

\begin{remark} We can think of  a multiple $\G$-path as a generalized map  $\I\overset{\e}{\gets}{\I'_S} \overset{\sigma}{\to}\G$ since the map $\e$ defined in the obvious way is an essential equivalence. Then, two multiple $\G$-paths $\I\overset{\e}{\gets}{\I'_S} \overset{\sigma}{\to}\G$  and $\I\overset{\e'}{\gets}{\I'_{S'}} \overset{\sigma'}{\to}\G$ are {\it homotopic} if there is a commutative diagram:
$$
\xymatrix{ & 
{\I'_S}\ar[dr]^{\sigma}="0" \ar[dl]_{\e }="2"&\\
{\I}&{\I'_{S''}} \ar[u]_{u} \ar[d]^{v}&{\G}\\
&{\I_{S'}}\ar[ul]^{\e'}="3" \ar[ur]_{\sigma}="1"&
\ar@{}"0";"1"|(.4){\,}="7" 
\ar@{}"0";"1"|(.6){\,}="8" 
\ar@{=>}"7" ;"8"_{H'} 
\ar@{}"2";"3"|(.4){\,}="5" 
\ar@{}"2";"3"|(.6){\,}="6" 
\ar@{=>}"5" ;"6"^{H} 
}
$$ 
where $S\cup S'\subset S''$
 \end{remark}

\section{Lusternik-Schnirelmann category}

\subsection{Categorical sets}
A submanifold $U\subset G_0$ is {\it invariant} if $t(s^{-1}(U))\subset U$.  The {\it restricted groupoid $\U$} to an invariant submanifold $U\subset G_0$ is the full groupoid over $U$. In other words, $U_0=U$ and $U_1=\{g\in G_1|s(g), t(g) \in U\}$. We write $\U=\G|_U\subset \G$.

 A restricted groupoid $\G|_{\cO}$ over an orbit $\cO$ will be called an {\it orbit groupoid}  and denoted by $\cO^K$, where $K=G_x$ is the isotropy group of $x$ for any $x\in \cO$.

A {\it generalized constant map}  from $\cK$ to $\G$ is a generalized map  $\displaystyle\cK\xleftarrow{\e}\cK'\xrightarrow{c}\G$ such that the image of the functor $c$ is contained in an orbit groupoid. We have that $c(\cK')\subset \cO^K$ with $K=G_x$ for some $x\in G_0$.

The {\it restriction} $(\e,\phi)|_{\V}$ of a generalized map  $\cK\overset{\e}{\gets}\cJ\overset{\phi}{\to}\G$ to the restricted groupoid $\V\subset \cK$ is the composition of the generalized map 
$(\e,\phi)$ and  the inclusion functor $\V\overset{\id}{\gets}\V\overset{i_{\V}}{\to}\cK$:
\[
\xymatrix@!=1pc{&&\cJ\times_{\cK}\V\ar[dr]^{p_1}\ar[dl]_{p_3}&&\\
&\V\ar[dl]_{\id}\ar[dr]^{i_{\V}}&&\cJ\ar[dl]_{\e}\ar[dr]^{\phi}&\\
\V &&\cK &&\G}
\]
where $\cJ\times_{\cK}\V$ is the fibered product groupoid.

The {\it product} $(\e,\phi)\times(\e',\phi')$ of two generalized maps $\cK\overset{\e}{\gets}\cJ\overset{\phi}{\to}\G$ and $\cK'\overset{\e}{\gets}\cJ'\overset{\phi'}{\to}\G'$ is given by the generalized map
$$\cK\times \cK'\overset{\e \times \e'}{\gets}\cJ \times \cJ'\overset{\phi \times \phi'}{\to}\G \times \G'.$$

We will define now a notion of $\G$-contraction within the groupoid $\G$. We keep the classical Lusternik-Schnirelmann terminology of  {\it categorical} for contractible subspaces in a given space.

\begin{definition} \label{d} For an invariant open set $U\subset G_0$, we will say that the restricted groupoid $\cU$ is $\G$-{\it categorical} if the inclusion functor $i_{\U}\colon \U\to \G$ is Morita homotopic to a generalized constant map
$$
\xymatrix{ & 
{\U}\ar[dr]^{i_{\U}}="0" \ar[dl]_{\id }="2"&\\
{\U}&{\cL} \ar[u]_{u} \ar[d]^{v}&{\G}\\
&{\U'}\ar[ul]^{\e}="3" \ar[ur]_{c}="1"&
\ar@{}"0";"1"|(.4){\,}="7" 
\ar@{}"0";"1"|(.6){\,}="8" 
\ar@{=>}"7" ;"8"_{H'} 
\ar@{}"2";"3"|(.4){\,}="5" 
\ar@{}"2";"3"|(.6){\,}="6" 
\ar@{=>}"5" ;"6"^{H} 
}
$$ 
where $u$ and $v$ are essential homotopy equivalences.
\end{definition}

\begin{example} Let $\G=\Hol(K,\F_S)$ be the holonomy groupoid of the Seifert fibration $\F_S$ on the Klein bottle $K$. Consider the invariant open set $U=M$, a M\"{o}bius band. Then the restricted groupoid over $U$ is $\U=\Hol(M,\F_S)$. Consider the essential homotopy equivalence $\eta: \bullet^{\Z_2} \to \Hol(M,\F_S)$ given by the inclusion as in example \ref{ex}. Let ${\Z_2}\ltimes I$ be the action groupoid given by the action of ${\Z_2}$ on the interval $I$ by reflection.
The inclusion functor $\e: {\Z_2}\ltimes I\to \Hol(M,\F_S)$ is an essential equivalence. We have the following diagram
$$
\xymatrix{ & 
{\Hol(M,\F_S)}\ar[dr]^{i_{\U}}="0" \ar[dl]_{\id }="2"&\\
{\Hol(M,\F_S)}&\bullet^{\Z_2}  \ar[u]_{\eta} \ar[d]^{i}&{\Hol(K,\F_S)}\\
&{\Z_2}\ltimes I\ar[ul]^{\e}="3" \ar[ur]_{c}="1"&
\ar@{}"0";"1"|(.4){\,}="7" 
\ar@{}"0";"1"|(.6){\,}="8" 
\ar@{=>}"7" ;"8"_{H'} 
\ar@{}"2";"3"|(.4){\,}="5" 
\ar@{}"2";"3"|(.6){\,}="6" 
\ar@{=>}"5" ;"6"^{H} 
}
$$ 
here $c: {\Z_2}\ltimes I\to {\Hol(K,\F_S)}$ is a constant map on objects $c(t)=x$ where $x$ is any point on a fiber with holonomy $\Z_2$ and $c$ is a constant map in each of the connected components of the manifold of arrows $\displaystyle I\sqcup I$, $c(t,0)=(x, 0)$ and $c(t,1)=(x, \pi)$. Here the manifold of arrows for the groupoid ${\Hol(K,\F_S)}$ is $K\times S^1$. Then $c({\Z_2}\ltimes I)\subset \bullet^{\Z_2}$ and $U=M$ is categorical.
\end{example} 
Given a categorical subgroupoid $\U$, we have that the inclusion $i_{\U}$ composed with an essential homotopy equivalence $u$ factors through an orbit groupoid up to homotopy:
$$
\xymatrix{  
{\U}\ar[dr]^{i_{\U}}="0" &\\
{\cL} \ar[u]_{u} \ar[d]^{v}&{\G}\\
{\U'}\ar[ur]_{c}="1" \ar[r]^{}&{\cO^K}\ar[u]^{}
\ar@{}"0";"1"|(.4){\,}="7" 
\ar@{}"0";"1"|(.6){\,}="8" 
\ar@{=>}"7" ;"8"_{} 
}
$$ 
In other words, if $\U$ is $\G$-categorical then there exists a groupoid $\cL$ and a group $K$ such that the diagram 

$$
\xymatrix{
{\cL}\ar[r]^{u}\ar[rd]^{}&{\U} \ar[r]^{i_{\U}}&{\G}\\
&{\cO^K}\ar[ur]^{}&
}
$$ 
commutes up to  homotopy.
 
Recall that a Lie groupoid is {\it transitive} if the map $(s,t):G_1\to G_0\times G_0$ is a surjective submersion.

\begin{prop} \cite{MM} A Lie groupoid $\G$ is Morita equivalent to a point groupoid if and only if $\G$ is transitive.
\end{prop}
    
Since an orbit groupoid is transitive, for $\U$ to be $\G$-categorical we can substitute the constant generalized map in definition \ref{d} by a generalized map with image contained in a point groupoid.
 
\begin{definition}    The LS category, $\cat \G$, of a Lie groupoid $\G$ is the least number of
    $\G$-categorical subgroupoids required to cover $\G$. If $\G$ cannot be covered by a finite number of $\G$-categorical subgroupoids, set $\cat \G=\infty$.
    \end{definition}

This number is an {\it invariant of  Morita equivalence} that generalizes the ordinary LS category of a manifold. If $\G=u(M)$ is the unit groupoid, then 
$\cat \G=\cat M$, where $\cat$ in the second term means the ordinary LS category.

With these definitions we choose the {\it point groupoids} as our contractible groupoids. We make this choice in spite of the fact that the fundamental group of the point groupoid is not necessary trivial since $\pi_1(\po)=K$ if $K$ is discrete. Our choice is inspired by the $\cA$-category of Clapp and Puppe \cite{Monica} and justified by the equivariant LS theory for group actions. The point groupoid can be seen as the translation groupoid of the (ineffective) action of $K$ on a point $\bullet$. Since the equivariant LS category of a $K$-manifold, $\cat_KM$, coincides with the ordinary category $\cat M$ when the action is trivial, we have that the equivariant category $\cat_K\bullet=\cat\bullet=1$. In our theory, a $\G$-categorical groupoid factors through a point groupoid. For all groups $K$ we have:
$$\cat\po=\cat_K\bullet=\cat\bullet=1$$ where the first term is the {\it groupoid category}  of the point groupoid, the second is the {\it equivariant category} of the action of $K$ on a point, and the third, the ordinary category of a  point.

In section \ref{6} we will give a weighted version of the LS groupoid category that takes into account the fact that point groupoids are not contractible in the sense mentioned above.

\subsection{Invariance of Morita homotopy type}
We will show that the LS category of a groupoid is an invariant of Morita homotopy type, and then, invariant under Morita equivalence.

\begin{prop}\label{morita}    If $\cK\simeq G$ then $\cat\cK=\cat\G$.
 \end{prop}
\begin{proof}
We will prove that if $\cK$ dominates $\G$, then $\cat\cK\ge\cat\G$. Let $\cK\overset{\e}{\gets}\cJ\overset{\phi}{\to}\G$ and $\G\overset{\delta}{\gets}\cJ'\overset{\psi}{\to}\cK$ be two generalized maps such that the composition $(\phi,\e)\circ(\psi,\delta)$ is Morita homotopic to the identity in $\G$. We have the following diagram:

\begin{equation}
\label{1}
\xymatrix@!=1pc{ 
&& {\G}\ar[ddrr]^{\id}="0" \ar[ddll]_{\id}="2"&&\\
&&&&\\
{\G}\rrtwocell<\omit>{<0>S}&&{\cE} \ar[uu]_{} \ar[dd]^{}\rrtwocell<\omit>{<0>S}&&{\G}\\
&\cJ'\ar[ul]^{\delta}&&\cJ\ar[ur]^{\phi}&\\
&&{\cJ\times_{\cK}\cJ'}\ar[ul]^{p_3}="3" \ar[ur]_{p_1}="1"&&
}
\end{equation}

Let $\U\subset \cK$ be a $\cK$-categorical subgroupoid, then we have the following diagram:

\begin{equation}
\label{2}
\xymatrix@!=1pc{
&& {\U}\ar[drr]^{i_{\cK}}="0" \ar[dll]_{\id}="2"&&\\
{\U}\rrtwocell<\omit>{<0>S'}&&{\cL} \ar[u]_{} \ar[dd]^{}\rrtwocell<\omit>{<0>S'}&&{\cK}\\
&&&\po\ar[ur]&\\
&&{\U'}\ar[uull]^{z}="3" \ar[ur]_{}="1"&&
}
\end{equation}

Let $\V\subset \G$ be the groupoid given by $\delta p'_3(\U\times_{\cK}\cJ')$ in the following pullback:

$$
\xymatrix{&
{\U\times_{\cK}\cJ'}\ar[r]^{p'_1}\ar[d]^{p'_3}&{\U} \ar[d]^{i_{\U}}\\
\G&{\cJ'}\ar[r]^{\psi}\ar[l]^{\delta}&\cK
}
$$

In diagram (\ref{1}), take the restriction to $\V$, i.e. compose both maps with the inclusion functor $\V\overset{\id}{\gets}\V\overset{i_{\V}}{\to}\G$

\begin{equation}
\label{3}
\xymatrix@!=1pc{ 
&&{\V}\ar[dr]^{i_{\V}} \ar[dl]_{\id}&&\\
&{\V}\ar@{=}[dd]\ar[dr]^{i_{\V}}\ar[dl]_{\id}&&{\G} \ar[dr]^{\id} \ar[dl]_{\id}&\\
\V&&{\G}\rtwocell<\omit>{<0>S}&\cE\rtwocell<\omit>{<0>S}\ar[u]_{}\ar[d]_{}&\G\\
&{\V}\ar[ur]^{i_{\V}}\ar[ul]^{\id}&&{\cJ\times_{\cK}\cJ'} \ar[ur]_{\phi p_1} \ar[ul]^{\delta p_3}&\\
&& {\cJ\times_{\cK}\cJ'\times_{\G}\V}\ar[ur]^{} \ar[ul]_{}&&
}
\end{equation}

Now take the restriction of $\G\overset{\delta}{\gets}\cJ'\overset{\psi}{\to}\cK$ to $\V$

\begin{equation}
\label{4}
\xymatrix@!=1pc{
&&\cJ'\times_{\G}\V\ar[dr]^{p_1}\ar[dl]^{p_3}&&\\
&\V\ar[dl]_{\id}\ar[dr]^{i_{\V}}&&\cJ'\ar[dl]_{\delta}\ar[dr]^{\psi}&\\
\V &&\G &&\cK}
\end{equation}

\begin{lemma} $\psi p_1(\cJ'\times_{\G}\V)\subset \U$.
\end{lemma}
\begin{proof} (of the lemma) Use that $U$ is invariant and the essential equivalence $\delta$ induces a diffeomorphism between the isotropy groups.
\end{proof}

Since $\psi p_1(\cJ'\times_{\G}\V)\subset \U$ we obtain the generalized map
$\V\overset{}{\gets}\cJ'\times_{\G}\V\overset{}{\to}\U$ from diagram (\ref{4}).

Now we compose this map with the inclusion functor $\U\overset{\id}{\gets}\U\overset{}{\to}\cK$ and
 with the generalized map $\cK\overset{\e}{\gets}\cJ\overset{\phi}{\to}\G$, we have:
 
 \begin{equation}
\label{5}
\xymatrix@!=1pc{
&&&&\cJ\times_{\cK}\U\times_{\U}\cJ'\times_{\G}\V\ar[drr]\ar[dll]&&\\
&&\U\times_{\U}\cJ'\times_{\G}\V \ar[dl]_{}\ar[dr]^{}&&&&\cJ\ar[ddll]_{\e}\ar[ddrr]^{\phi}&&\\
&\cJ'\times_{\G}\V \ar[dl]_{}\ar[dr]^{}&&\U\ar[dl]_{}\ar[dr]^{}&&&&&\\
\V&&\U&&\cK&&&&\G
}
\end{equation}
We will show that the generalized map $\V\overset{}{\gets}\cJ\times_{\cK}\cJ'\times_{\G}\V\overset{}{\to}\G$ obtained by this composition is homotopic to the map 
$$\V\overset{}{\gets}\cJ\times_{\cK}\U'\times_{\U}\cJ'\times_{\G}\V\overset{c'}{\to}\G$$ obtained from the following compositions

\begin{equation}
\label{6}
\xymatrix@!=1pc{
&&&&\cJ\times_{\cK}\U'\times_{\U}\cJ'\times_{\G}\V\ar[drr]^{\pi_1}\ar[dll]&&\\
&&\U'\times_{\U}\cJ'\times_{\G}\V \ar[dl]_{}\ar[dr]^{}&&&&\cJ\ar[ddll]_{\e}\ar[ddrr]^{\phi}&&\\
&\cJ'\times_{\G}\V \ar[dl]_{}\ar[dr]^{}&&\U'\ar[dl]_{z}\ar[dr]^{c}&&&&&\\
\V&&\U&&\cK&&&&\G
}
\end{equation}
Let $c'=\phi \pi_1$, we have the following 
\begin{lemma}\label{lema} The image of $c'=\phi \pi_1$  is contained in an orbit groupoid. Moreover, if $c(\U')\subset {\cO}^K$  then  $c'(\cJ\times_{\cK}\U'\times_{\U}\cJ'\times_{\G}\V)\subset {\cO}^{K'}$ where $K\approx K'$.
\end{lemma}
\begin{proof} (of the lemma) Use the fact that $\e$ induces a  homeomorphism between the orbit spaces and a diffeomorphism between the isotropy groups.
\end{proof}

Then $\V\overset{}{\gets}\cJ\times_{\cK}\U'\times_{\U}\cJ'\times_{\G}\V\overset{c'}{\to}\G$ is a generalized constant map. By using diagram (\ref{2}) we have that 
\begin{equation}
\label{7}
\V\overset{}{\gets}\cJ\times_{\cK}\cJ'\times_{\G}\V\overset{}{\to}\G\qquad\simeq_{S'}\qquad \V\overset{}{\gets}\cJ\times_{\cK}\U'\times_{\U}\cJ'\times_{\G}\V\overset{c'}{\to}\G
\end{equation}

Then we have a 2-morphism 

\begin{equation}
\label{8}
\xymatrix{ & 
{\V}\ar[dr]^{}="0" \ar[dl]_{}="2"&\\
{\V}&{\cL} \ar[u]_{} \ar[d]^{}&{\G}\\
&{\cJ\times_{\cK} \cJ'\times_{\G}\V}\ar[ul]^{}="3" \ar[ur]_{}="1"&
\ar@{}"0";"1"|(.4){\,}="7" 
\ar@{}"0";"1"|(.6){\,}="8" 
\ar@{=>}"7" ;"8"_{S} 
\ar@{}"2";"3"|(.4){\,}="5" 
\ar@{}"2";"3"|(.6){\,}="6" 
\ar@{=>}"5" ;"6"^{S} 
}
\end{equation}
given by diagram (\ref{3}) and another 2-morphism 

\begin{equation}
\label{9}
\xymatrix{ & 
{{\cJ\times_{\cK} \cJ'\times_{\G}\V}}\ar[dr]^{}="0" \ar[dl]_{}="2"&\\
{\V}&{\cL'} \ar[u]_{} \ar[d]^{}&{\G}\\
&{\cJ\times_{\cK}\U'\times_{\U}\cJ'\times_{\G}\V}\ar[ul]^{}="3" \ar[ur]_{c'}="1"&
\ar@{}"0";"1"|(.4){\,}="7" 
\ar@{}"0";"1"|(.6){\,}="8" 
\ar@{=>}"7" ;"8"_{S'} 
\ar@{}"2";"3"|(.4){\,}="5" 
\ar@{}"2";"3"|(.6){\,}="6" 
\ar@{=>}"5" ;"6"^{S'} 
}
\end{equation}
given by equation (\ref{7}).

 The vertical composition of these 2-morphisms (\ref{8}) and (\ref{9}) gives a 2-morphism between the inclusion and a generalized constant map:
 
\[
\xymatrix{ 
&& {\V}\ar[ddrrr]^{i_{\V}}="2"  \ar[ddll]_{\id}="0" &&&\\
&&\cL \ar[dr]^{v}\ar[u]_{}&&&\\
\cK& \rtwocell<\omit>{<0>\;\;\scriptstyle S''} &\cP_{S''}\rtwocell<\omit>{<0>\;\;\scriptstyle S''}\ar[d]^{} \ar[u]_{}&\cD\rtwocell<\omit>{<0>\;\;\scriptstyle S''}&&\G\\
&&\cL'\ar[d]^{} \ar[ur]_{u'} &&&\\
&&\cJ\times_{\cK}\U'\times_{\U}\cJ'\times_{\G}\V\ar[uull]^{}="1"  \ar[uurrr]_{}="3" &&&&
} 
\] 
where $\cD=\cJ\times_{\cK} \cJ'\times_{\G}\V$. Therefore, $\V$ is categorical for $\G$.
\end{proof}

\subsection{Homotopy restrictions}
If $\U$ is $\G$-categorical, we have that the following diagram commutes up to homotopy:
$$
\xymatrix{ & 
{\U}\ar[dr]^{i_{\U}}="0" \ar[dl]_{\id }="2"&\\
{\U}&{\cL} \ar[u]_{u} \ar[d]^{v}&{\G}\\
&{\U'}\ar[ul]^{\e}="3" \ar[ur]_{c}="1"\ar[r]^{}&{\po}\ar[u]^{}
\ar@{}"0";"1"|(.4){\,}="7" 
\ar@{}"0";"1"|(.6){\,}="8" 
\ar@{=>}"7" ;"8"_{H'} 
\ar@{}"2";"3"|(.4){\,}="5" 
\ar@{}"2";"3"|(.6){\,}="6" 
\ar@{=>}"5" ;"6"^{H} 
}
$$ 
Then, for all $y\in L_0$  the isotropy group $G_y$ injects into $G_x$ for  $x=u(y)\in U$ and into $G_z$ for  $z=v(y)\in U'$
 by proposition \ref{inj} and remark \ref{isotropy}. We have that $G_x \rightarrowtail K$. In particular, if the isotropy groups are finite, we have that 
 $|G_x|$ divides $|K|$ for all $x\in U$. For instance, a categorical subgroupoid $\U$ cannot factor through a trivial group $K$ except in case that all the points in $U$ have trivial isotropy.

 \begin{example} Let $\G=\Hol(K,\F_S)$ be the holonomy groupoid of the Seifert fibration $\F_S$ on the Klein bottle $K$. We can cover $K$ by two M\"{o}bius
bands $M$ as before. Since the isotropy at the points in the center fiber is ${\Z_2}$  these points cannot be moved by a homotopy to a point in the neighborhood with trivial isotropy. Therefore this covering is minimal and the groupoid category is $\cat \Hol(K,\F_S)=2$.
\end{example}

\section{Orbifolds as groupoids}
We recall now the description of orbifolds as groupoids due to Moerdijk and Pronk \cite{MP,M}. Orbifolds were first introduced by Satake \cite{S} as a generalization of a manifold defined in terms of local quotients. The groupoid approach provides a global language to reformulate the notion of orbifold.

We follow the exposition in \cite{M} and \cite{A}. A groupoid $\G$ is {\it proper} if $(s,t):G_1\to G_0\times G_0$ is a proper map and it is a {\it foliation} groupoid if each isotropy group is discrete. 

\begin{definition}
An {\it orbifold} groupoid is a proper foliation groupoid.
\end{definition}

For instance the holonomy group of a foliation $\F$ is always a foliation groupoid but it is an orbifold groupoid if and only if $\F$ is a compact-Hausdorff foliation.

Given an orbifold groupoid $\G$, its orbit space $|\G|$ is a locally compact Hausdorff space. Given an arbitrary locally compact Hausdorff space $X$ we can equip it  with an orbifold structure as follows:

\begin{definition} An {\it orbifold structure} on a locally compact Hausdorff space $X$ is given by an orbifold groupoid $\G$ and a homeomorphism $h:|\G|\to X$.
\end{definition}
If $\e:\cH\to \G$ is an essential equivalence and $|\e|:|\cH|\to |\G|$ is the induced homeomorphism between orbit spaces, we say that the composition $h\circ|\e|:|\cH|\to X$ defines an {\it equivalent} orbifold structure. 

\begin{definition}  An {\it orbifold}  $\X$ 
is a space $X$ equipped with an equivalence class of orbifold 
structures. A specific such structure, given by 
$\G$ and  $h : |\G | \to X $ is
a {\it presentation} of the orbifold 
$\X$.
\end{definition}

If two groupoids are Morita equivalent, then they define the same orbifold. Therefore any structure or invariant for orbifolds, if defined through groupoids, should be invariant under Morita equivalence. 

\begin{definition}  An {\it orbifold map} $f\colon \Y\to \X$ is given by an equivalence class of generalized maps $(\e,\phi)$ from $\cK$ to $\G$ between presentations of the orbifolds such that the diagram commutes:

\[\xymatrix{
|\cK| \ar[r]^{|\phi||\e|^{-1}}\ar[d]& |\G| \ar[d]^{} \\ 
Y \ar[r]& X}\]
A specific such generalized map $(\e,\phi)$ is called a {\it presentation} of the orbifold map $f$.
\end{definition}
Our  notion of Morita homotopy gives a notion of homotopy for orbifolds:
\begin{definition}
An {\it orbifold homotopy} between the orbifold maps $f, g\colon \Y\to \X$ is given by a Morita homotopy between the presentations 
 $(\e,\phi)$ of $f$ and $(\delta,\psi)$ of $g$.
\end{definition}
In other words, an orbifold homotopy is given by the following diagram 
 $$
\xymatrix{ & 
{\cK'}\ar[dr]^{\phi}="0" \ar[dl]_{\eta}="2"&\\
{\cK}&{\cL} \ar[u]_{u} \ar[d]^{v}&{\G}\\
&{\cK''}\ar[ul]^{\nu}="3" \ar[ur]_{\psi}="1"&
\ar@{}"0";"1"|(.4){\,}="7" 
\ar@{}"0";"1"|(.6){\,}="8" 
\ar@{=>}"7" ;"8"_{H'} 
\ar@{}"2";"3"|(.4){\,}="5" 
\ar@{}"2";"3"|(.6){\,}="6" 
\ar@{=>}"5" ;"6"^{H} 
}
$$ 
where $\cK$ to $\G$ are presentations of the orbifolds $\Y$ and $\X$.

\begin{remark} Since the groupoid $\cL$ has the same {\it Morita homotopy type} as $\cK$, we can always take a presentation of the orbifold homotopy $\Y\times \I\to \X$ given by a strict homotopy $H^S\colon \cL\times \I_S\to \G$, where $\cL$ and $\cK$ define the same orbifold homotopy class of orbifolds.
\end{remark}

\begin{example}\label{ex1}
Consider the orbifold $\X$ having as a presentation groupoid the holonomy groupoid associated to the Seifert fibration on the M\"{o}bius band, $\G=\Hol(M,\F_S)$, its underlying space is $X=[0,1)$. The orbifold $\X$ has the same orbifold homotopy type as  the orbifold $\Y$ represented by the point groupoid $\cK=\bullet^{\Z_2}$. The Morita homotopy equivalence is given by an orbifold map $ \Y\to \X$ with presentation:
$$\bullet^{\Z_2}\overset{\eta}{\gets}\Z_2\ltimes I\overset{\nu}{\to}\Hol(M,\F_S)$$
as in example \ref{mht}. The homotopic inverse $ \X\to \Y$ is given by the generalized map $\Hol(M,\F_S)\overset{\nu}{\gets}\Z_2\ltimes I\overset{\eta}{\to}\bullet^{\Z_2}$.

\end{example}

\subsection{LS category for orbifolds}
If $\G$ is a presentation for the orbifold $\X$, we define the LS category of $\X$ as the groupoid category of $\G$, $\cat \X=\cat \G$.

By proposition \ref{morita} the LS category of a groupoid is invariant under Morita equivalence, then this definition is independent of the chosen presentation for $\X$.

The LS category of an orbifold $\X$ is an invariant of the orbifold homotopy type which generalizes the classical LS  category in case that the orbifold is a manifold. If the orbifold is {\it effective} the LS category for orbifolds as groupoids coincide with the LS category for orbifolds as defined by the author in \cite{Skye} using local charts.

\begin{remark} Since a transitive groupoid is Morita equivalent to a point groupoid, we have that $\cat \X=1$ if  the underlying topological space $X=|\G|$ of the orbifold $\X$ is a point.
\end{remark}

For instance, we have that in example \ref{ex1} the orbifold category of $\X$ is $\cat \X=1$ since the presentation $\G=\Hol(M,\F_S)$ has the same Morita homotopy type of a point groupoid. In this case, the ordinary category of the underlying topological space $X=[0,1)$, coincides with the orbifold category of $\X$, though in general, it is just a lower bound:

\begin{prop}
Let $\X$ be an orbifold and $X$ its underlying topological space. Then $\cat\X\ge\cat X$. 
\end{prop}
\begin{proof}
Let $\G$ be a presentation groupoid for the orbifold $\X$, then $|\G|$ is homeomorphic to $X$. We will show that $\cat\G\ge\cat |\G|$. Let $U\subset G_0$ be an invariant open set such that $\U$ is $\G$-categorical. Since $s$ and $t$ are open maps, we have that $\pi\colon \G\to |\G|$ is an open map too. Therefore, $\pi(U)$ is an open set in $|\G|$. The $S$-homotopy between the inclusion and the constant map induces a homotopy $H\colon \pi(U)\times |\I_S | \to |\G|$ in the orbit spaces. Since $\I_S$ is Morita equivalent to the interval $I$ regarded as a unit groupoid, we have that  $|\I_S | =I$. Thus $\pi(U)$ is categorical for  $|\G|=X$, in the ordinary sense.
\end{proof}

\begin{example}
 The groupoids $\G$ and $\cK$ in example \ref{tear} define the {\it teardrop} orbifold $\X$. For the groupoid $\G$, we have a $\G$-categorical covering given by $U=D_1$ and $V=D_2$ and for $\cK$ we can consider the $\cK$-categorical covering given by $U=S^3-T_N$ and $V=S^3-T_S$ where $T_S$  and $T_N$ are the cores of the solid tori. We have that $\cat \X=2$.
\end{example}

We recall now the notion of inertia groupoid \cite{M} and  its decomposition  in twisted and untwisted sectors.
We show that the orbifold category of the twisted sectors provides a lower bound for the orbifold category of the orbifold.

Let $\G$ be an orbifold groupoid. The inertia groupoid $\wedge \G$ is a Lie groupoid whose manifold of objects is given by
$$(\wedge \G)_0=S_{\G}=\{g\in G_{1} | s(g)=t(g)\}.$$
We can think of the elements  in $S_{\G}$ as loops in $\G$. The manifold of arrows 
$$(\wedge \G)_1=S_{\G}\times_{G_0}^s G_1$$ is given by the following pullback of manifolds:
\[\xymatrix{
S_{\G}\times_{G_0}^s G_1 \ar[r]^{} \ar[d]_{}& S_{\G}\ar[d]^{\beta} \\ 
G_1\ar[r]^{s} & G_0}\]
where  $\beta$ sends  a loop $g:x\to x$ to its source (or target)  $\beta(g)=x$. Then an arrow between the loop $g:x\to x$ and $g':y\to y$ is given by an arrow $h:x\to y$ such that $hgh^{-1}=g'$.
The source and target maps $s,t\colon S_{\G}\times_{G_0}G_1\to  S_{\G}$ are given by $s(g,h)=g$ and $t(g,h)=hgh^{-1}$.

If $\cK\simM \G$, then $\wedge\cK\simM \wedge\G$ and there is a well defined notion of inertia for orbifolds. The orbifold $\wedge\X$ is determined by the inertia groupoid $\wedge\G$ corresponding to any presentation $\G$ of $\X$.

If $\G=G\ltimes M$ is a translation groupoid, then the inertia groupoid is also a translation groupoid given by 
\[(\wedge \G)_0=\{(g,x)\,|\,x\in M, g\in G, gx=x\}\] and 
the action of $G$  given by $h\cdot(g,x)=(hx, hgh^{-1})$.

 \begin{example} \label{dihedral} 
 \begin{enumerate}
 \item The inertia orbifold for the teardrop orbifold $\X$ in example \ref{tear} is given by the disjoint union of itself and $2$ copies of the point $N$. Then the orbifold category of the inertia orbifold is given by $cat \wedge\X =\cat \X + 2=4$.
 
 \begin{figure}
\includegraphics[height=1in]{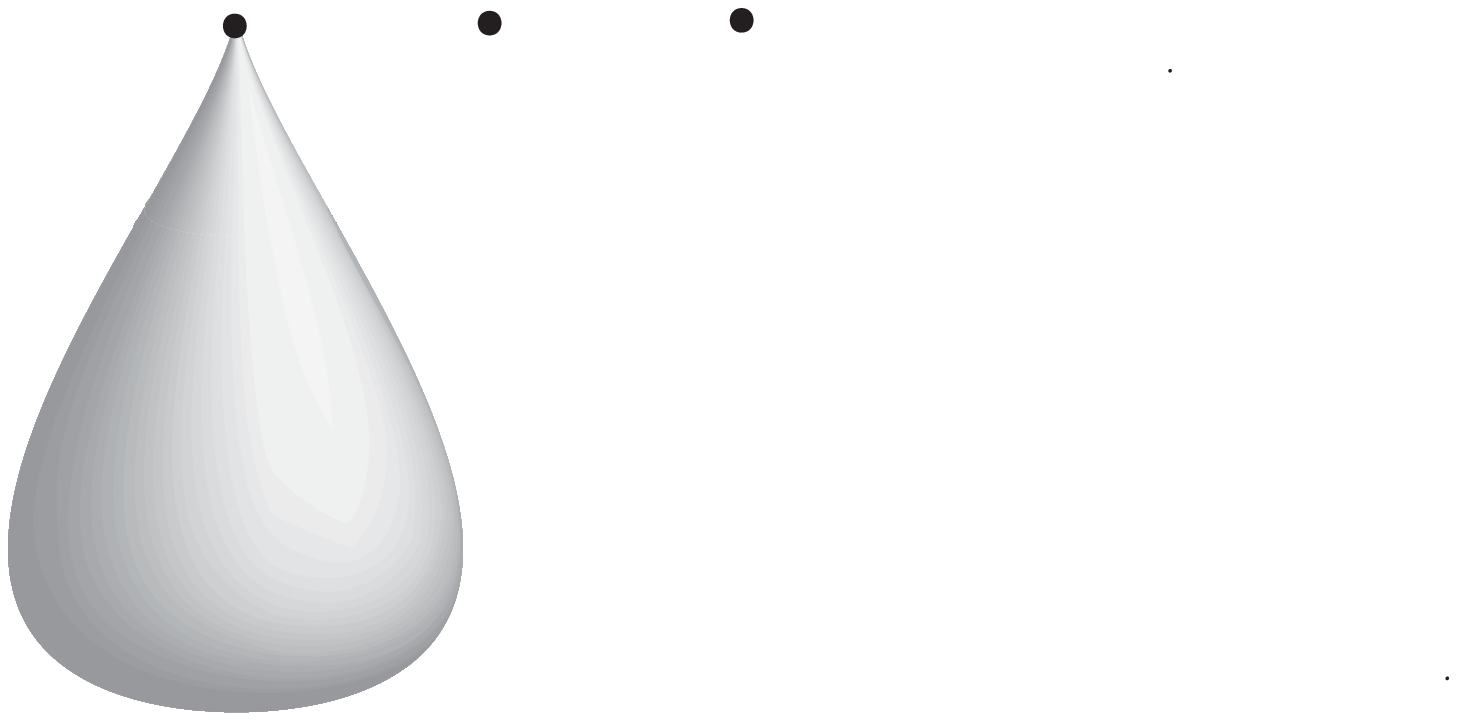}
\caption{}
\label{5}
\end{figure}

\item 
Given  the action 
    of the dihedral group $$D_8=\{1,\rho, \rho^2,  \rho^3, \sigma, \sigma 
    \rho, \sigma\rho^2, \sigma\rho^3\}$$ on the sphere $S^2$ where $\sigma$ 
    acts by reflection and $\rho$ by rotation about the $z$-axis, consider the translation groupoid $\G=D_8\ltimes S^2$ defining the orbifold $\X$. Both poles $N$ and $S$ in $G_0=S^2$ are fixed by the action, then they have isotropy $G_N=G_S=D_8$. There are four great circles whose points have isotropy $\Z_2$ generated by each of the symmetries  $\sigma, \sigma 
    \rho, \sigma\rho^2$ and $\sigma\rho^3$.
     The manifold of objects $S_{\G}$ of the inertia groupoid $\wedge \G$ is given by the disjoint union of $S^2$ plus $3$ copies of each $N$ and $S$ corresponding to the rotations $\rho, \rho^2$ and $ \rho^3$ and the  four great circles with non-trivial isotropy as shown in figure \ref{D}

   \begin{figure}
\includegraphics[height=1in]{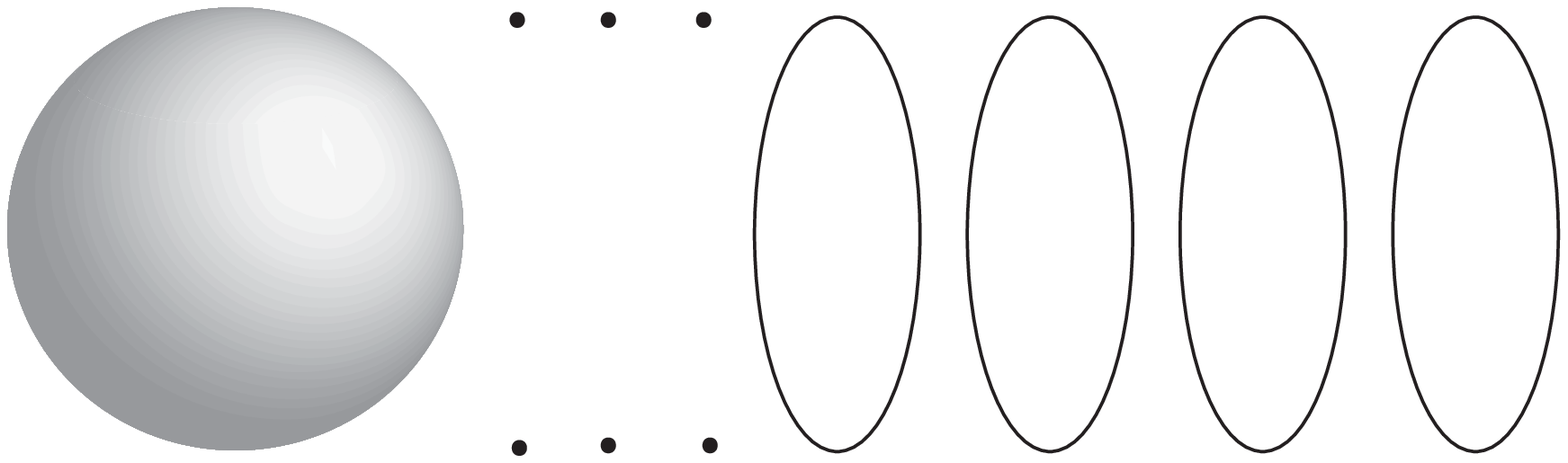}
\caption{}
\label{D}
\end{figure}
    
 The orbifold $\X$ defined by $\G$ is a disk orbifold with silvered boundary and two dihedral singular points on the boundary. The inertia orbifold $\wedge \X$ given by the action of $G=D_8$ on $S_{\G}$ is the disjoint union of the orbifolds shown in figure \ref{D4}:
    \begin{figure}
\includegraphics[height=1in]{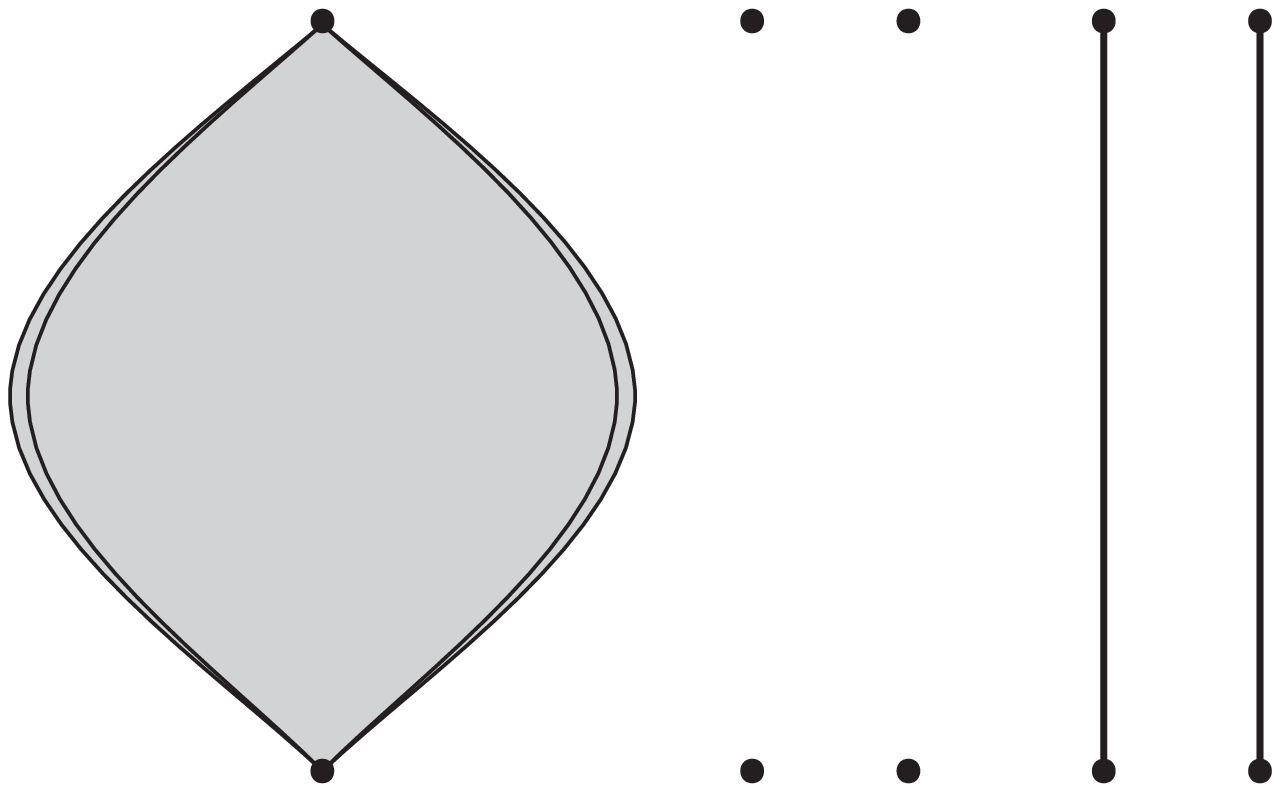}
\caption{}
\label{D4}
\end{figure}
 Consider the covering $\{\U, \V\}$ by subgroupoids of $\G$ given by $U=S^2-N$ and $V=S^2-S$. This covering is categorical and $\cat \X=2$. The category of each of the  1-dimensional orbifolds in the decomposition is $2$. Then the category of the inertia orbifold is $\cat \wedge \X=\cat \X+4+2+2=10$.

\end{enumerate}
 \end{example}

Let $G^i$ be a $\G$-saturated connected component in $S_{\G}$. If $\G^i$ is the full subgroupoid over $G^i$, we have that the partition of $S_{\G}=\bigsqcup G^i$ induces a decomposition of the inertia groupoid:
$$\wedge\G =\bigsqcup \G^i$$ as well as a decomposition of the inertia orbifold, $\wedge\X =\bigsqcup \X^i$. The components $G^i\subset S_{\G}$ besides $G_0$ are called  {\it twisted sectors} whilst $G_0$ is the {\it untwisted sector}. The corresponding $\G^i\subset \G$ and $\X^i\subset \X$ are also called twisted sectors. 

\begin{prop} If  $\X^i$ is a twisted sector of the orbifold $\X$, then $\cat\X\ge \cat\X^i$. 
\end{prop}
\begin{proof}
If $\U$ is a $\G$-categorical subgroupoid for $\G$, consider the subgroupoid $\V=p_1(\G^i\times_{\G} \U)$ in $\G^i$ given by the following fibered product of groupoids:
$$ \xymatrix{
 \G^i\times_{\G}\U \ar[d]_{} \ar[r]^{p_1}&\G^i \ar[d]^{\beta_i}\\
 \U \ar[r]_{i_{\U}}& \G }$$
where $\beta_i$ is the restriction to $\G^i$ of $\beta\colon \wedge \G\to \G$ given by $\beta(g)=s(g)$ on objects and $\beta(g,h)=h$ on arrows. The inclusion functor $i_{\U}\colon \U\to \G$ is Morita homotopic to a generalized constant map $\displaystyle\U\xleftarrow{\e}\U'\xrightarrow{c}\G$. Consider the fibered product $\V'=\G^i\times_{\G} \U'$. We have that the inclusion $i_{\V}\colon \V\to \G^i$ is Morita homotopic to a generalized constant map $\displaystyle\V\xleftarrow{}\V'\xrightarrow{}\G^i$.
\end{proof}

\subsection{Weighted category}\label{6}

In the classical Lusternik-Schnirelmann theory for topological spaces, we count the amount of categorical sets required to cover the space. There are no distinctions between the categorical sets since all of them factor through a point. In the groupoid context, we count subgroupoids which factor through point groupoids.
 We are going to distinguish our point groupoids by their {\it weights} and count categorical subgroupoids together with  the weights of their associated point groupoids. This idea is inspired by the work of Tom Leinster on the Euler characteristic of a finite category \cite{Tom} and based on the string-theoretic orbifold Euler characteristic \cite{Wi,At,adem}.
 
 Let $\G$ be an orbifold groupoid and $\U$ be a $\G$-categorical subgroupoid. Let $K$ be a group of {\it minimal order} such that $\U$ 
 factors through $\po$. We define the weight of $\U$ as the number $k$ of conjugacy classes in $K$, $w(\U)=k$. If $K$ is abelian, then the weight of $\U$ is the order of $K$,  $w(\U)=|K|$.

A  covering $\{\U_1,\ldots, \U_n\}$ of $\G$ is {\it minimal } if $\U_i$ is $\G$-categorical and $\cat \G=n$. Given a minimal covering $\{\U_1,\ldots, \U_n\}$ and its associated set of weights $\{w_1,\ldots, w_n\}$ we propose the following 

\begin{definition}    The {\it weighted} LS category, $w\cat \G$, of a Lie groupoid $\G$ is the least value of the sum:
$$\sum_{i=1}^n w_i$$ for all  $\{\U_1,\ldots, \U_n\}$ minimal coverings.

 \end{definition}
 
 \begin{example}
 \begin{enumerate}
 \item
Consider the translation groupoid $\G=S^1\ltimes S^3$ defining the teardrop orbifold $\X$ as in example \ref{tear}. 
For the categorical covering $\{\U, \V\}$ by subgroupoids of $\G$ given by $U=S^3-T_N$ and $V=S^3-T_S$ we have that the associated set of weights is $\{1, 3\}$. This covering is minimal and gives the least sum, then $w\cat \G=4$.
\item For  the translation groupoid  $\G=D_8\ltimes S^2$, consider the categorical covering $\{\U, \V\}$ by subgroupoids of $\G$ given by $U=S^2-N$ and $V=S^2-S$ as in example \ref{dihedral} (2). Both subgroupoids $\U$ and $\V$ factor through the group $D_8$. The conjugation classes in $D_8$ are given by $\{\{1\}, \{\rho, \rho^3\},\{\rho^2\},\{\sigma, \rho^2\sigma\},\{\rho\sigma, \rho^3\sigma\}\}$, then the class number is $5$.
Therefore, a the set of weights associated to this covering is $\{5, 5\}$. We have that $w\cat \G=10$.
\end{enumerate}
 \end{example}

We will show that the weighted category of an orbifold groupoid is an invariant of Morita homotopy type, then in particular yields a well defined invariant for orbifolds: $w\cat \X=w\cat \G$ for any groupoid presentation $\G$. Moreover, the weighted category for orbifolds is an invariant of orbifold homotopy type.

 \begin{prop} If $\cK$ and $\G$ are orbifold groupoids with $\cK\simeq\G$, then $w\cat \cK=w\cat \G$.
  \end{prop} 
  \begin{proof}
  Let $w\cat \cK=m$ and let $\{\U_1,\ldots, \U_n\}$ be a minimal $\cK$-covering realizing the least sum, then $m=w_1+\ldots +w_n$ where $\{w_1,\ldots, w_n\}$  is the set of weights associated to $\{\U_1,\ldots, \U_n\}$. We construct a $\G$-minimal covering $\{\V_1,\ldots, \V_n\}$ as in the proof of Proposition \ref{morita}. By Lemma \ref{lema} we have that $\{w_1,\ldots, w_n\}$  is also a set of weights for $\{\V_1,\ldots, \V_n\}$ and it is minimal with this property. Then  $w\cat \G=w_1+\ldots +w_n=m$.
  \end{proof}
  
  We conjecture that the weighted category of an orbifold groupoid coincides with the (unweighted) category of its inertia groupoid, $w\cat \G=\cat \wedge \G$.

\section{Critical points}
Classically, the LS category of a manifold is a lower bound for the number of critical points of a smooth function under certain conditions. We propose a generalization of the notion of critical point for generalized maps and prove that the LS theorem holds  for {\it orbifold groupoids}.

Let  $\phi:\cL\to \G$ be a morphism of Lie groupoids. A point $x\in L_0$ is a {\it critical point} for $\phi$ if $x$ is critical for the object part of the functor, $(d\phi_0)_x=0$.

\begin{prop}\label{p1}
If  $x\in L_0$ is a critical point for $\phi$, then:
\begin{enumerate}
\item for all $y\in \cO(x)$, $y$ is a critical point for $\phi$ and 
\item for all $g\in L_1$ with $s(g)=x$,  $g$ is a critical point (arrow) for $\phi_1$
\end{enumerate}
\end{prop}
\begin{proof}
Let $g\in L_1$ be an arrow $g\colon x\to y$. From the following diagram 
\[\xymatrix{
L_1 \ar[r]^{\phi_1} \ar[d]_s \ar@<1ex>[d]^t& G_1\ar[d]_s \ar@<1ex>[d]^t\\ 
L_0 \ar[r]^{\phi_0} & G_0}\]
we have that $d(s\phi_1)_g=d(\phi_0s)_g$. Then 
$$(ds)_{\phi_1(g)}\circ (d\phi_1)_g=(d\phi_0)_{s(g)}\circ (ds)_g=0$$
since $x=s(g)$ is critical for $\phi_0$. Since $s$ is a submersion, we have that  $(ds)_{\phi_1(g)}$ has maximal rank and then $(d\phi_1)_g=0$. Thus, $g$ is  critical for $\phi_1$. Moreover, this implies that $(dt)_{\phi_1(g)}\circ (d\phi_1)_g=0$. Since $d(t\phi_1)_g=d(\phi_0t)_g$, then $(d\phi_0)_{t(g)}\circ (dt)_g=0$. We have that $(d\phi_0)_{t(g)}=0$ since $t$ has maximal rank. Then $y=t(g)$ is critical for $\phi_0$.
\end{proof}

Therefore, if $x$ is a critical point for $\phi$, we will write $(d\phi)_{\cO^K}=0$ where $\cO$ is the orbit of $x$ and $K$ the isotropy group at $x$.

\begin{definition}  Let $\phi:\cL\to \G$ be a morphism of Lie groupoids. An orbit groupoid $\cO^K\subset \cL$ is a {\it critical subgroupoid} if $(d\phi)_{\cO^K}=0$.
\end{definition}

\begin{prop} \label{p2}
Let $\phi, \psi:\cL\to \G$ be two equivalent morphisms, $\phi\sim_T\psi$. The orbit groupoid $\cO^K$ is a critical subgroupoid for $\phi$ if and only if $\cO^K$ is a critical subgroupoid for $\psi$.
\end{prop}
\begin{proof} If $T\colon L_0\to G_1$ is a natural transformation between $\phi$ and $\psi$,  the following diagrams are commutative
\[\xymatrix{
& G_1\ar[d]_s \\ 
L_0 \ar[r]^{\phi_0}\ar[ru]^{T} & G_0}\qquad\xymatrix{
& G_1\ar[d]_t \\ 
L_0 \ar[r]^{\psi_0}\ar[ru]^{T} & G_0}\]
and we have that $d(sT)_x=d\phi_x$ and $d(tT)_x=d\psi_x$. If $x$ is critical for $\phi$, then $ds_{T(x)}\circ dT_x=0$ and since $s$ has maximal rank, $dT_x=0$. Then $d\psi_x=dt_{T(x)}\circ dT_x=0$ and $x$ is critical for $\psi$.
\end{proof}

\begin{prop} \label{p3}
If $\e:\cL\to \cK$ is an essential equivalence, then $\e$ has no critical subgroupoids.
\end{prop}
\begin{proof} An essential equivalence $\e$ is equivalent by a natural transformation to $\e'$ with $\e'_0$ being a submersion.
\end{proof}

\begin{definition} An orbit groupoid $\cO^K\subset \cJ$ is a {\it critical subgroupoid for the generalized map $\cJ\overset{\e}{\gets}\cL\overset{\phi}{\to}\G$} if there exists a critical subgroupoid $\cO'^{K'}\subset \cL$ for $\phi$ such that $\e(\cO'^{K'})\subset \cO^K$.
\end{definition}
In this case, the groups $K$ and $K'$ are isomorphic.

We will show that the notion of critical subgroupoid is invariant under Morita equivalence and yields a notion of critical point for orbifold morphisms.
\begin{prop} \label{p4}
Let $(\e,\phi)$ and $(\e',\phi')$ be two equivalent generalized maps. The orbit subgroupoid $\cO^K\subset \cJ$ is critical for $(\e,\phi)$ if and only if it is critical for $(\e',\phi')$.
\end{prop}
\begin{proof} 
Since the  following diagram commutes up to natural transformations
$$
\xymatrix{ & 
{\cL}\ar[dr]^{\phi}="0" \ar[dl]_{\e }="2"&\\
{\cJ}&{\cA} \ar[u]_{\alpha} \ar[d]^{\beta}&{\G}\\
&{\cL'}\ar[ul]^{\e'}="3" \ar[ur]_{\phi'}="1"&
\ar@{}"0";"1"|(.4){\,}="7" 
\ar@{}"0";"1"|(.6){\,}="8" 
\ar@{}"7" ;"8"_{\sim_{T}} 
\ar@{}"2";"3"|(.4){\,}="5" 
\ar@{}"2";"3"|(.6){\,}="6" 
\ar@{}"5" ;"6"^{\sim_{T'}} 
}
$$
we have that $(d\phi\alpha)_S=(d\phi'\beta)_S$ for all orbit subgroupoids $S \subset \cA$. Since $\cO^K\subset \cJ$ is critical for $(\e,\phi)$, we have that there exists a subgroupoid $\cO'^{K'}\subset \cL$ for $\phi$ with $\e(\cO'^{K'})\subset \cO^K$ and $d\phi_{{\cO'}^{K'}}=0$. 

Consider the subgroupoid $S=p_3(\cO'^{K'}\times_{\cL} \cA)$ given by the following fibered product of groupoids:
\[\xymatrix{
\cO'^{K'}\times_{\cL} \cA \ar[r]^{p_1}\ar[d]_{p_3}& \cO'^{K'}\ar[d]^{i_{\cO'^{K'}} }\\ 
\cA \ar[r]^{\alpha} & \cL}\]
Then $(d\phi)_{\alpha(S)}\circ d\alpha_S=(d\phi')_{\beta(S)}\circ d\beta_S=0$ and $(d\phi')_{\beta(S)}=0$ since $\beta$ has maximal rank. Taking $\cO''^{K''}=\beta(S)\subset \cL'$ we have that  $\cO^K$ is critical for $(\e',\phi')$.
\end{proof}

Then the notion of critical subgroupoid is well defined for the class of equivalence of the generalized map defining the orbifold morphism.

\subsection{Relative groupoid LS-category}
As in the classical case, we develop a {\it relative} version of the Lusternik-Schnirelmann groupoid category.

\begin{definition} Let $\G$ be a Lie groupoid and  $M\subset G_0$ be an invariant open set. Consider the full subgroupoid $\M$ over $M$. The {\it groupoid LS-category of $\M$ in $\G$, $\cat(\M, \G)$, } is the least number of $\G$-categorical subgroupoids required to cover $\M$.
\end{definition}

Let $\M$ and $\cN$ be two full subgroupoids of $\G$.

\begin{definition} The subgroupoid $\M$ is $\G$-deformable into $\cN$ in $\G$ if there are groupoids $\M'$ and $\cL$ such that the following diagram commutes:
$$
\xymatrix{ & 
{\M}\ar[dr]^{i_{\M}}="0" \ar[dl]_{\id}="2"&\\
{\M}&{\cL} \ar[u]_{u} \ar[d]^{v}&{\G}\\
&{\M'}\ar[ul]^{\eta}="3" \ar[ur]_{\phi}="1"\ar[r]_{}&\cN\ar[u]_{i_{\cN}}
\ar@{}"0";"1"|(.4){\,}="7" 
\ar@{}"0";"1"|(.6){\,}="8" 
\ar@{=>}"7" ;"8"_{H^S} 
\ar@{}"2";"3"|(.4){\,}="5" 
\ar@{}"2";"3"|(.6){\,}="6" 
\ar@{=>}"5" ;"6"^{F^S} 
}
$$
\end{definition}

\begin{remark} A subgroupoid $\U$ is $\G$-categorical if $\U$ is $\G$-deformable into an orbit groupoid $\cO^K$ in $\G$.
\end{remark}

The relative groupoid LS-category for orbifold groupoids satisfy the following properties.
\begin{enumerate}
\item[P1] If $\M\subset\cN$, then $\cat(\M, \G)\le\cat(\cN, \G)$.
\item[P2] $\cat(\M\cup \cN, \G)\le\cat(\M, \G)+\cat(\cN, \G)$.
\item[P3] If $\M$ is $\G$-deformable into $\cN$, then $\cat(\M, \G)\le\cat(\cN, \G)$.
\item[P4] If $W\subset G_0$ is an invariant closed set, then there exists an invariant open set $U$ with $W\subset U$ such that $\cat(\U, \G)=\cat(\cW, \G)$. We will say that the subgroupoid $\U$ is a neighborhood of $\cW$.
\end{enumerate}
In particular, every orbit subgroupoid in an orbifold groupoid has a $\G$-categorical neighborhood. Note that this property does not hold for general Lie groupoids.
\begin{remark} If $\G$ is an orbifold groupoid defining the orbifold $\X$, then for each $x\in G_{0}$ there exist arbitrary small
neighborhoods $U$ of $x$ for which $\U$ is isomorphic to the action groupoid $G_{x}\ltimes U$  \cite{MP}. The subgroupoid $\U$ determines an open set
$|\U|\subseteq |\G|=X$, for which $|\U|$ is the quotient of $U$ by the action of the finite group $G_{x}$. 
\end{remark}

\subsection{Critical subgroupoid $\cK$}

Let $\phi:\G\to \R$ be a morphism of Lie groupoids, where $\R$ is the unit groupoid over the real numbers. Let $\cK$ be the critical subgroupoid of $\phi$, i.e. the subgroupoid given by the union of all  critical orbit subgroupoids: $$\cK=\bigsqcup_{(d\phi)_{\cO^K}=0} \cO^K$$

Let $\phi(K_0)$ be the set of critical values and $R=\R\setminus \phi(K_0)$ the set of regular values. Consider the full subgroupoids  $\G_c=\phi^{-1}(-\infty, c]$ and $\cK_c=\cK\cap \phi^{-1}(c)$. We will impose the following $\G$-deformation conditions on the groupoid $\G$ and the morphism $\phi$:
\begin{enumerate}
  \item[D1] For all $c$ in the interior of the set of regular values $R$ there exists an $\e>0$ such that $\G_{c+\e}$ is $\G$-deformable into $\G_{c-\e}$.
  \item [D2] For any isolated critical value $c$ and $\U$ neighborhood of $\cK_c$, there is an $\e>0$ such that  $\G_{c+\e}\setminus \U$ is $\G$-deformable into $\G_{c-\e}$.
  \item [D3] If $c>\sup \phi(\cK)$ then there is an $\e>0$ such that $\G$ is $\G$-deformable into $\G_{c}$.
\end{enumerate}

We develop a version of the Lusternik-Schnirelmann theorem for the groupoid LS category following the modern approach to the classical theorem by Clapp and Puppe \cite{Monica}.
\begin{prop}[Lusternik-Schnirelmann theorem]
Let $\G$ be an orbifold groupoid and $\phi\colon \G\to \R$ a morphism satisfying the $\G$-deformation conditions. Consider the function $m\colon \R\to \N\cup \{ \infty \}$ given by $m(c)=\cat(\G_c, \G)$. We assert that
\begin{enumerate}
\item the function $m$ is (weakly) increasing;
\item $m$ is locally constant in the interior of the set of regular values $R$;
\item at any isolated critical value $c$ of $\phi$ the function $m$ jumps by $\cat(\cK_c, \G)$ at most and
\item when  $c>\sup \phi(\cK)$ then $m(c)=\cat\G$.
\end{enumerate}
\end{prop}
As in the classical case, the theorem follows from properties P1-P4 and conditions D1-D3.
\begin{corollary}\label{LS} If $\G$ is an orbifold  groupoid and $\phi:\G\to \R$ a morphism satisfying the $\G$-deformation conditions, then $$\cat \G\le \sum_{c\in\R} \cat(\cK_c, \G).$$
\end{corollary} 
This implies that  $\cat \G $ is a lower bound for the  number of critical orbits of $\phi$.

Given an orbifold $\X$ and an orbifold map $\X\to \R$, we will say that a point $x\in\X$ is {\it critical} if there is a presentation groupoid $\G$ for $\X$ and a presentation generalized map $\G\overset{\e}{\gets}\G'\overset{\phi}{\to}\R$ for  $\X\to \R$ such that $x$ is  the image in the quotient $\G\to |\G|=X$  of some critical orbit subgroupoid $\cO^K\subset \G$ for the generalized morphism  $(\e,\phi)$.

We claim the following version of the Lusternik-Schnirelmann theorem for orbifolds.

\begin{theorem}
Let $\X$ be a compact orbifold and $\X\to \R$ be an orbifold map. Then $cat \X$ is a lower bound for the number of critical points of  $\X\to \R$.
\end{theorem}

To prove that an orbifold groupoid $\G$ and a morphism $\phi$ satisfy the $\G$-deformation conditions we will introduce the notion of {\it integral $\G$-curve for $\G$-vector fields}.

We define the {\it tangent groupoid} $T\G$ of the groupoid $\G$ as a groupoid whose manifold of objects and arrows  are $TG_0$ and $TG_1$ respectively. Source and target are given by the differentials of $s,t\colon G_1\to G_0$ which are also submersions: 
\[\xymatrix{
TG_1\ar[d]_{ds} \ar@<1ex>[d]^{dt}\\ 
TG_0}\]

\begin{definition}
A  $\G$-{\it vector field} is a morphism $X\colon \G\to T\G$ such that $X_0\colon G_0\to TG_0$ and $X_1\colon G_1\to TG_1$ are vector fields.
\end{definition}

We will use here the idea of multiple $\G$-paths to define $\G$-curves.

\begin{definition} A $\G$-{\it curve} is a morphism $\sigma\colon \cJ'_S \to \G$ where $\cJ'_S$ is the groupoid given by a subdivision $S=\{r_0\le r_1\le\cdots\le r_n\}$ of an open interval $J$ as defined before.

We write $\sigma=(\al^{j_1}_1,g_1, \al^{j_2}_2 \ldots, \al^{j_{n-1}}_{n-1}, g_{n-1}, \al^{j_n}_n)$ where $\al^{j_1}_1$ and $\al^{j_n}_n$ are maps from half-closed intervals to $G_0$ and $(g_1,\al^{j_2}_2 \ldots, \al^{j_{n-1}}_{n-1}, g_{n-1})$ is a multiple $\G$-path between the orbits of $\al^{j_1}_1(r_0)$ and $\al^{j_n}_n(r_n)$. 
\end{definition}

The {\it velocity} of $\sigma$ is a $\G$-curve $\dot{\sigma}\colon \cJ'_S \to T\G$ given by $$\sigma=(\dot{\al}^{j_1}_1, dg_1, \dot{\al}^{j_2}_2\ldots, \dot{\al}^{j_{n-1}}_{n-1}, dg_{n-1}, \dot{\al}^{j_n}_n)$$
where $dg_i$ maps the velocity vector $\dot{\al}_i(r_i,i,1)$ to the velocity vector $\dot{\al}^1_{i+1}(r_i,i+1,1)$.

The {\it length} $l$ of a $\G$-curve $\sigma$ is the sum of the lengths of the paths $\al^1_i$. We have that 
$$l(\sigma)=\sum_1^n\int_{r_i}^{r_{i+1}}|\dot{\al}_i(r,i,1)|$$
The sum of the lengths of any other branch of the multiple $\G$-path will give the same length. The {\it distance} between two orbits $\cO$ and $\cO'$ is the infimum of the the lengths of $\G$-paths between the orbifold subgroupoids $\cO^K$ and ${\cO'}^{K'}$. This distance defines a metric which is complete if the orbifold is compact.

Let $X$ be a $\G$-vector field on $\G$.  We will say that $\sigma$ is an {\it integral $\G$-curve} for $X$ if $\dot{\sigma}=X\circ \sigma$. If $0$ is in $J$, we call $\sigma({\bf{0}})$ the initial subgroupoid of the integral $\G$-curve.

For a compact orbifold groupoid, we have that for each orbit subgroupoid $\cO^K$ there is a maximal integral $\G$-curve $\sigma_{\cO^K}$   for $X$ whose domain is given by $J=(-\infty, +\infty)$ and the initial subgroupoid is $\cO^K$.

\begin{definition}
The {\it flow} of $X$ is the morphism $\varphi\colon \G\times \R'_S\to \G$ given by $\varphi(x,(r,i,j))=\sigma_{\cO^K}((r,i,j))$ with $x\in\cO$ on objects and $\varphi(g,r_i)=\sigma_{\cO^K}(r_i)$ with $g\in K$ on arrows.
\end{definition}
We have that $\varphi(x,(r'+r,i+i',j))=\varphi(\varphi(x,(r,i,j)),(r',i',j)))$ on  objects and similarly on arrows.

Given a morphism $\phi\colon\G\to \R$ where $\G$ is an orbifold groupoid defining a compact orbifold, consider the $\G$-vector field $X=\nabla\phi$ given by the gradient $\nabla \phi\colon \G\to T\G$.  Then  the flow of  the $\G$-vector field $X=\nabla\phi$ gives the Morita homotopy required by the $\G$-deformation conditions.
Then an orbifold groupoid $\G$ defining a compact orbifold $\X$  is in the hypothesis of corollary \ref{LS} and the statement of the theorem follows.

\end{document}